\documentclass{amsart}

\usepackage{ae,aecompl}
\usepackage{esint}
\usepackage{graphicx}
\usepackage[all,cmtip]{xy}
\usepackage{amsmath, amscd}
\usepackage{booktabs}
\usepackage{amsthm}
\usepackage{amssymb}
\usepackage{amsfonts}
\usepackage{qsymbols}
\usepackage{latexsym}
\usepackage{mathrsfs}
\usepackage{cite}
\usepackage{color}
\usepackage{url}
\usepackage{enumerate}
\usepackage{undertilde}

\usepackage{verbatim}
\usepackage[draft=false, colorlinks=true]{hyperref}

\allowdisplaybreaks

\newcommand {\bmo}{{\mathrm{bmo}}}
\newcommand {\BMO}{{\mathrm{BMO}}}

\newcommand {\C}{{\mathbb C}}

\newcommand {\ud}{\mathrm{d}}

\newcommand {\veps}{\varepsilon}

\newcommand {\Ell}{L}

\newcommand {\F}{{\mathcal{F}}}

\newcommand {\Hrinf}{\mathcal{H}^{r,\infty}}
\newcommand {\HT}{\mathcal{H}}
\newcommand {\Hp}{\mathcal{H}^{p}_{FIO}(\Rn)}
\newcommand {\Hps}{\mathcal{H}^{s,p}_{FIO}(\Rn)}
\newcommand {\Hpt}{\mathcal{H}^{t,p}_{FIO}(\Rn)}

\newcommand {\rb}{\rangle}
\newcommand {\lb}{{\langle}}
\newcommand {\La}{{\mathcal{L}}}

\newcommand {\N}{{{\mathbb N}}}

\newcommand {\ph}{{\varphi}}

\newcommand {\R}{{\mathbb R}}

\newcommand {\Rn}{{\mathbb{R}^{n}}}

\newcommand {\supp}{{\mathrm{supp}}}

\newcommand {\Sw}{\mathcal{S}}

\newcommand {\w}{{\omega}}

\newcommand {\Z}{{{\mathbb Z}}}

\newcommand {\vanish}[1]{\relax}

\newcommand{\wh}{\widehat}
\newcommand{\wt}{\widetilde}

\DeclareFontFamily{U}{mathx}{\hyphenchar\font45}
\DeclareFontShape{U}{mathx}{m}{n}{
      <5> <6> <7> <8> <9> <10>
      <10.95> <12> <14.4> <17.28> <20.74> <24.88>
      mathx10
      }{}
\DeclareSymbolFont{mathx}{U}{mathx}{m}{n}
\DeclareFontSubstitution{U}{mathx}{m}{n}
\DeclareMathAccent{\widecheck}{0}{mathx}{"71}

\newtheorem{theorem}{Theorem}[section]
\newtheorem{lemma}[theorem]{Lemma}
\newtheorem{proposition}[theorem]{Proposition}
\newtheorem{corollary}[theorem]{Corollary}

\theoremstyle{definition}
\newtheorem{definition}[theorem]{Definition}
\newtheorem{remark}[theorem]{Remark}

\numberwithin{equation}{section}
\protected\def\ignorethis#1\endignorethis{}
\let\endignorethis\relax

\title[Rough pseudodifferential operators on Hardy spaces for FIOs]{Rough pseudodifferential operators on Hardy spaces for Fourier integral operators}

\author{Jan Rozendaal}
\address{Institute of Mathematics, Polish Academy of Sciences\\
ul.~\'{S}niadeckich 8\\
00-656 Warsaw\\
Poland}
\email{jrozendaal@impan.pl}



\thanks{This research was supported by ARC grant DP160100941, and partially supported by NCN grant UMO2017/27/B/ST1/00078. The research leading to these results has received funding from the Norwegian Financial Mechanism 2014-2021, grant 2020/37/K/ST1/02765.}

\begin{document}

\begin{abstract}
We prove mapping properties of pseudodifferential operators with rough symbols on Hardy spaces for Fourier integral operators. The symbols $a(x,\eta)$ are elements of $C^{r}_{*}S^{m}_{1,\delta}$ classes that have limited regularity in the $x$ variable. We show that the associated pseudodifferential operator $a(x,D)$ maps between Sobolev spaces $\Hps$ and $\Hpt$ over the Hardy space for Fourier integral operator $\Hp$. Our main result implies that for $m=0$, $\delta=1/2$ and $r>n-1$, $a(x,D)$ acts boundedly on $\Hp$ for all $p\in(1,\infty)$.
\end{abstract}

\maketitle

\section{Introduction}

\subsection{Setting}
Pseudodifferential operators have long been a powerful tool for the analysis of partial differential equations. For example, they appear as parametrices for various equations. They also arise naturally through paradifferential calculus, where one collects the frequencies of a differential operator into groups. In fact, the combination of these two tools, parametrices and paradifferential calculus, has shown itself to be particularly potent when studying equations with rough coefficients. Most relevant for the present article are wave equations with rough coefficients, which arise for example from Laplace--Beltrami operators on manifolds with a rough metric, or by linearizing nonlinear wave equations with rough initial data. 

Given a differential operator $A$ with rough coefficients, one can apply a paradifferential smoothing procedure that goes back to Bony~\cite{Bony81} (see also e.g.~Meyer~\cite{Meyer81a,Meyer81b} and Taylor~\cite{Taylor91,Taylor00}) to decompose $A$ as a sum of a smooth pseudodifferential operator $A_{1}$ and a rough pseudodifferential operator $A_{2}$ of lower order. Microlocal parametrices can then be constructed for the pseudodifferential equation $(\partial_{t}^{2}-A_{1})u(t)=0$, and by obtaining suitable mapping properties for $A_{2}$ one can use Duhamel's principle to obtain a parametrix for the rough wave equation $(\partial_{t}^{2}-A)u(t)=0$ as well. The resulting parametrix then contains a lot of information about the rough equation, while also being more tractable than the bona fide solution operators to the equation. A specific instance of this paradigm was developed by Smith in~\cite{Smith98b}, and subsequently applied by both Smith and Tataru to obtain powerful results for wave equations with rough coefficients, such as Strichartz estimates~\cite{Smith98b,Tataru00,Tataru01,Tataru02}, propagation of singularities~\cite{Smith14}, the related spectral cluster estimates~\cite{Smith06}, and well-posedness of nonlinear wave equations with rough initial data \cite{Smith-Tataru05}.

In these applications, it suffices to obtain mapping properties between Sobolev spaces over $L^{2}(\Rn)$ for the rough pseudodifferential term. Such mapping properties are classical, even in the more general setting of $L^{p}(\Rn)$-based Sobolev spaces \cite{Bourdaud82,Marschall88,Taylor91}. However, when considering the fixed-time $L^{p}$ regularity of wave equations with rough coefficients, it no longer suffices to obtain mapping properties in the $L^{p}(\Rn)$ scale, as will be explained next.

It has long been known that even smooth wave equations do not preserve the $L^{p}$ regularity of initial data. In fact, for the classical wave propagator one has 
\begin{equation}\label{eq:classicalwave}
e^{it\sqrt{-\Delta}}:W^{2s(p),p}(\Rn)\to L^{p}(\Rn)
\end{equation}
for $1<p<\infty$, $t\in\R$ and $s(p):=\frac{n-1}{2}|\frac{1}{p}-\frac{1}{2}|$, as was shown by Peral~\cite{Peral80} and Miyachi~\cite{Miyachi80a}, and the exponent $2s(p)$ cannot be improved. General wave equations with smooth coefficients have the same fixed-time $L^{p}$ regularity, as is a consequence of the work of Seeger, Sogge and Stein from 1991~\cite{SeSoSt91} on the $L^{p}$ regularity of Fourier integral operators. Moreover, for $p=1$ and $p=\infty$ one obtains corresponding estimates upon replacing $L^{p}(\Rn)$ by the local Hardy space $\HT^{1}(\Rn)$ and $\bmo(\Rn)$, respectively. To the author's best knowledge, since then, essentially no progress has been made on determining the optimal $L^{p}$ regularity for a general class of wave equations with rough coefficients.  

One reason for the dearth of results on the fixed-time $L^{p}$ regularity for wave equations with rough coefficients might be that, in general, iterative approximation arguments are not useful for wave equations on the $L^{p}(\Rn)$ scale for $p\neq 2$. Indeed, due to the unboundedness of Fourier integral operators on $L^{p}(\Rn)$, approximation arguments that involve iterative constructions typically lead to a loss of regularity in each iteration step, and such constructions only work for infinitely smooth initial data. However, in~\cite{Smith98a} Smith introduced a powerful replacement of the local Hardy space $\HT^{1}(\Rn)$ that is adapted to Fourier integral operators. More precisely, his space $\HT^{1}_{FIO}(\Rn)$ is invariant under suitable Fourier integral operators of order zero, and in particular under smooth wave propagators, and it satisfies Sobolev embeddings that allow one to recover the results of Seeger, Sogge and Stein.

With an eye towards the fixed-time $L^{p}$ regularity of wave equations with rough coefficients, in~\cite{HaPoRo20} Hassell, Portal and the author extended Smith's construction to a scale $(\Hp)_{1\leq p\leq \infty}$ of Hardy spaces for Fourier integral operators. These spaces are invariant under compactly supported Fourier integral operators of order zero associated with a canonical transformation, and they satisfy the Sobolev embeddings 
\begin{equation}\label{eq:Sobolev1}
W^{s(p),p}(\Rn)\subseteq \Hp\subseteq W^{-s(p),p}(\Rn)
\end{equation}
for $1<p<\infty$, with appropriate modifications for $p=1$ and $p=\infty$. By considering the Sobolev space $\HT^{s(p),p}_{FIO}(\Rn)=\lb D\rb^{-s(p)}\Hp$ over $\Hp$, one directly recovers \eqref{eq:classicalwave} and the optimal $L^{p}$ regularity of Fourier integral operators. 

However, beyond merely recovering existing results, the Hardy spaces for Fourier integral operators allow for iterative constructions that were not available before. In particular, one can use the combination of paradifferential calculus and microlocal parametrices to prove that wave equations with rough coefficients are solvable over $\Hp$. Via \eqref{eq:Sobolev1}, one then obtains as a corollary the optimal $L^{p}$ regularity for such equations. For this strategy to work, one has to construct a parametrix for a smooth pseudodifferential equation on $\Hp$, and determine the mapping properties between Sobolev spaces over $\Hp$ of rough pseudodifferential operators. In the present article we consider the latter problem. The results from this article were subsequently used by Hassell and the author in the companion paper~\cite{Hassell-Rozendaal23} to obtain the first results on the optimal fixed-time $L^{p}$ regularity for a general class of wave equations with rough coefficients. We also note that in~\cite{Frey-Portal20}, invariant spaces for Fourier integral operators were used in a different manner to obtain results for wave equations with specific rough coefficients.

\subsection{Main results}
Our results are formulated in terms of a class $C^{r}_{*} S^{m}_{1,\delta}$ of rough symbols. Symbols $a(x,\eta)$ in $C^{r}_{*} S^{m}_{1,\delta}$ behave like elements of H\"{o}rmander's $S^{m}_{1,\delta}$ class, except that they have limited regularity in the $x$ variable (see Definition \ref{def:rough}). For $r\notin\N$ one requires $C^{r}(\Rn)$ regularity, while for $r\in\N$ slightly less is assumed. One associates with such a symbol $a$ the pseudodifferential operator $a(x,D)$ given by
\[
a(x,D)f(x)=\int_{\Rn}e^{ix\cdot\eta}a(x,\eta)\wh{f}(\eta)\ud\eta
\]
for suitable $f$. A special case of our main result is then as follows.

\begin{theorem}\label{thm:introduction}
Let $r>n-1$ and $a\in C^{r}_{*}S^{0}_{1,1/2}$. Then 
\[
a(x,D):\Hp\to \Hp
\]
is bounded for all $p\in(1,\infty)$. 
\end{theorem}

Theorem \ref{thm:introduction} is a consequence of Theorem \ref{thm:critical}, which also deals with the case where $r\leq n-1$, and which more generally considers mapping properties between Sobolev spaces $\Hps$ and $\Hpt$ over $\Hp$, for suitable $s$ and $t$. From this theorem one can in turn derive mapping properties on $\Hp$ for $C^{r}_{*}S^{0}_{1,\delta}$ symbols for general $\delta\in[0,1/2]$, as is shown in Corollary \ref{cor:combined}.

It should be noted that we consider $C^{r}_{*}S^{m}_{1,\delta}$ symbols with $\delta>0$ not merely to formulate more general statements. In fact, when applying the paradifferential smoothing procedure to a differential operator $A$, the resulting smooth pseudodifferential operator $A_{1}$ and rough pseudodifferential operator $A_{2}$ have $S^{m}_{1,\delta}$ and $C^{r}_{*}S^{m}_{1,\delta}$ symbols, respectively, and only $\delta>0$ is of interest here. Moreover, $\delta=1/2$ appears to be a critical value for wave equations, as it is for the more general theory of Fourier integral operators. We stress that Theorem \ref{thm:introduction} and the more general Theorem \ref{thm:critical} are crucial for the results on the $L^{p}$ regularity of wave equations with rough coefficients in~\cite{Hassell-Rozendaal23}, since they provide one of the two pieces that are required for the paradifferential approach to such equations.

However, the results in this article are of independent interest. For example, for $\delta=0$ they apply to what are seemingly the simplest pseudodifferential operators: multiplication operators. That $L^{p}(\Rn)$ is invariant under multiplication by bounded measurable functions is trivial, but it is nontrivial to determine the mapping properties of multiplication operators on the Hardy spaces for Fourier integral operators. This difficulty arises from the definition of the $\Hp$ norm as
\begin{equation}\label{eq:HpFIOnorm1}
\|f\|_{\Hp}=\|q(D)f\|_{L^{p}(\Rn)}+\Big(\int_{S^{n-1}}\|\ph_{\w}(D)f\|_{L^{p}(\Rn)}^{p}\ud\w\Big)^{1/p}
\end{equation}
for $1<p<\infty$, with a suitable modification for $p=1$ (see Definition \ref{def:HpFIO}). Here $q\in C^{\infty}_{c}(\Rn)$ is a low-frequency cutoff, and the Fourier multipliers $\ph_{\w}(D)$ localize in frequency to a paraboloid in the direction of $\w\in S^{n-1}$. Since Fourier multipliers and multiplication operators do not commute, it is not clear from \eqref{eq:HpFIOnorm1} how multiplication operators act on $\Hp$, and the situation becomes even more complicated when multiplying with a function of limited regularity. Nonetheless, Corollaries \ref{cor:subcritical} and \ref{cor:combined} show that for each $r>0$, multiplication with an element of $C^{r}_{*}(\Rn)$ is a bounded operator on $\Hps$ for all $p$ in an interval around $2$ and for suitable $s$, both depending on $r$ and the dimension $n$.

\subsection{Overview of the proof}
To prove our results we use various techniques. Firstly, mapping properties of rough pseudodifferential operators between Sobolev spaces over $L^{p}(\Rn)$ are classical~\cite{Bourdaud82,Marschall88, Taylor91,Taylor00}, and by combining these with the Sobolev embeddings in \eqref{eq:Sobolev1}, one obtains some first results. Moreover, it was shown in~ \cite{HaPoRo20} that pseudodifferential operators with $S^{0}_{1,1/2}$ symbols are bounded on $\Hp$ for all $p\in[1,\infty]$. For $\delta<1/2$ one can then apply the paradifferential smoothing procedure, this time to the rough symbol $a$ itself. As a result one obtains that, if $a\in C^{r}_{*}S^{0}_{1,\delta}$ for some $r>0$ and $\delta<1/2$, then $a(x,D)$ acts boundedly on $\Hps$ for all $p$ in an interval around $2$ and for suitable $s$, both depending on $r$, $n$ and $\delta$. 

However, for the critical case where $\delta=1/2$, such techniques do not suffice. Instead, to prove Theorem \ref{thm:introduction}, and the more general Theorem \ref{thm:critical}, we first note that \eqref{eq:HpFIOnorm1} allows us to apply Littlewood--Paley theory to the Hardy spaces for Fourier integral operators, by applying it to $\ph_{\w}(D)f$ for each $\w\in S^{n-1}$. In fact, it should be noted that \eqref{eq:HpFIOnorm1} is not how the $\Hp$ norm was originally defined in \cite{Smith98a,HaPoRo20}. It was shown by the author in \cite{Rozendaal21} that we may equivalently define the $\Hp$ norm in this manner for $1<p<\infty$, and the corresponding result for $p=1$ was obtained by Fan, Liu, Song and the author in~\cite{FaLiRoSo23}. 

With these equivalent characterizations and Littlewood--Paley theory in hand, we first remove the lowest and the highest frequencies of $a$. This reduces to the case where $\F a(\cdot,\eta)(\xi)=0$ unless $|\eta|^{1/2}\lesssim |\xi|\lesssim (1+|\eta|)^{\gamma}$, for some $\gamma\in[1/2,1)$. The highest frequencies are dealt with using the Sobolev embeddings in \eqref{eq:Sobolev1} and results about rough pseudodifferential operators on $W^{s,p}(\Rn)$ from~\cite{Marschall88}. On the other hand, the term with the lowest frequencies is an $S^{0}_{1,1/2}$ symbol, and the corresponding pseudodifferential operator is bounded on $\Hps$ for all $p\in[1,\infty]$ and $s\in\R$. The remaining frequencies, and specifically those $\xi$ for which $|\xi|\eqsim|\eta|^{1/2}$, are the most problematic, and dealing with these is the heart of the proof.

Next, we use a symbol decomposition from \cite{Coifman-Meyer78} to reduce to the case where $a(x,\eta)=\sum_{k=0}a_{k}(x)\psi_{k}(\eta)$ for suitable multiplication operators $a_{k}$ and dyadic localizations $\psi_{k}$. We then proceed to group together the frequencies of each $a_{k}$, as well as those of the function $f$ to which $a(x,D)$ is applied. Such a paradifferential approach has been used in e.g.~\cite{Bourdaud82,Marschall88,Taylor91,Taylor00} to study rough pseudodifferential operators on $L^{p}(\Rn)$. However, in those cases the frequencies are grouped together in dyadic annuli, whereas in our setting it is necessary to apply a further decomposition to the frequencies of $f$, the so-called dyadic-parabolic or second dyadic decomposition. The proof of \cite[Proposition 7.4]{Geba-Tataru07} is somewhat similar to this part of our argument, although there the authors obtain $L^{2}$ bounds.

We then use an anisotropic version of the Mikhlin multiplier theorem, Lemma \ref{lem:multiplier}, to estimate away the Fourier multipliers $\ph_{\w}(D)$, uniformly in $\w\in S^{n-1}$, after which we can use supremum norm estimates on $L^{p}(\Rn)$ to deal with the multiplication operators $a_{k}$. In the process we pick up terms that blow up if $r\leq n-1$, and this is where the restriction $r>n-1$ in Theorem \ref{thm:introduction} comes in. To conclude the proof we use the boundedness of the Hardy--Littlewood maximal function on $L^{p}(S^{n-1};\ell^{2})$. We note that the use of the anisotropic multiplier theorem and the boundedness of the Hardy--Littlewood maximal function are what restricts our proof to $p\in(1,\infty)$ (see Remark \ref{rem:p1}). We will not deal with the cases $p=1$ and $p=\infty$ in Theorem \ref{thm:critical} in the present article, although the results in Section \ref{sec:subcritical} hold for all $p\in[1,\infty]$.

\subsection{Organization of this article}
This article is organized as follows. In Section \ref{sec:preliminaries} we collect some background for the rest of the article. In Section \ref{subsec:Hardy} we first introduce the Hardy spaces for Fourier integral operators. Then we collect some results from Littlewood--Paley theory that will be used for the proof of Theorem \ref{thm:critical}, such as the anisotropic multiplier theorem in Lemma \ref{lem:multiplier}. In Section \ref{subsec:symbols} we introduce the rough symbol classes, as well as the paradifferential smoothing procedure. In Section \ref{sec:subcritical} we then derive some first results about the boundedness of rough pseudodifferential operators on Hardy spaces for Fourier integral operators. These results will suffice for many purposes, but not for $C^{r}_{*}S^{m}_{1,1/2}$ symbols. That critical case is dealt with in Section \ref{sec:critical}, where in particular Theorem \ref{thm:introduction} is proved.

\subsection{Notation}\label{subsec:notation}

The natural numbers are $\N=\{1,2,\ldots\}$, and the nonnegative integers are $\Z_{+}=\N\cup\{0\}$. Throughout this article, we fix $n\in\N$ with $n\geq2$. Our techniques can also be applied for $n=1$, but in that case the results are classical, by Proposition \ref{lem:pseudoLp} and because $\HT^{s,p}_{FIO}(\R)=\HT^{s,p}(\R)$ for all $p\in[1,\infty]$ and $s\in\R$, cf.~\eqref{eq:classical} and \eqref{eq:sobolev}. 

For $\xi\in\Rn$ we write $\lb\xi\rb=(1+|\xi|^{2})^{1/2}$, and $\hat{\xi}=\xi/|\xi|$ if $\xi\neq 0$. We use multi-index notation, where $\partial^{\alpha}_{\xi}=\partial^{\alpha_{1}}_{\xi_{1}}\ldots\partial^{\alpha_{n}}_{\xi_{n}}$ and $\xi^{\alpha}=\xi_{1}^{\alpha_{1}}\ldots\xi_{n}^{\alpha_{n}}$ for $\xi=(\xi_{1},\ldots,\xi_{n})\in\Rn$ and $\alpha=(\alpha_{1},\ldots,\alpha_{n})\in\Z_{+}^{n}$.

The spaces of Schwartz functions and tempered distributions are $\Sw(\Rn)$ and $\Sw'(\Rn)$, respectively.  The Fourier transform of $f\in\Sw'(\Rn)$ is denoted by $\F f$ or $\widehat{f}$. If $f\in\Ell^{1}(\Rn)$ then
\begin{align*}
\F f(\xi)=\int_{\Rn}e^{-i x\xi}f(x)\ud x\quad (\xi\in\Rn).
\end{align*}
The Fourier multiplier with symbol $\ph\in\Sw'(\Rn)$ is denoted by $\ph(D)$. 

The H\"{o}lder conjugate of $p\in[1,\infty]$ is denoted by $p'$. The volume of a measurable subset $B$ of a measure space $(\Omega,\mu)$ is $|B|$. For an integrable $F:B\to\C$, we write
\[
\fint_{B}F(x)\ud\mu(x)=\frac{1}{|B|}\int_{B}F(x)\ud\mu(x)
\]
if $|B|<\infty$. The space of bounded linear operators between Banach spaces $X$ and $Y$ is $\La(X,Y)$, and $\La(X):=\La(X,X)$. We will often simply write $\|\cdot\|_{F}$ for the norm of a function space $F(\Rn)$ over $\Rn$.

We write $f(s)\lesssim g(s)$ to indicate that $f(s)\leq Cg(s)$ for all $s$ and a constant $C>0$ independent of $s$, and similarly for $f(s)\gtrsim g(s)$ and $g(s)\eqsim f(s)$.

\section{Preliminaries}\label{sec:preliminaries}

In this section we collect the required background on the Hardy spaces for Fourier integral operators, rough symbol classes, and on the paradifferential smoothing procedure.

\subsection{Hardy spaces for Fourier integral operators}\label{subsec:Hardy}

In this subsection we introduce the Hardy spaces for Fourier integral operators, and we collect some of their basic properties.

Throughout, to simplify various statements, for $s\in\R$ and $p\in[1,\infty]$ we write 
\begin{equation}\label{eq:classical}
\HT^{s,p}(\Rn):=
\begin{cases}
W^{s,p}(\Rn)&\text{if }p\in(1,\infty),\\
\lb D\rb^{-s}\HT^{1}(\Rn)&\text{if }p=1,\\
\lb D\rb^{-s}\bmo(\Rn)&\text{if }p=\infty.
\end{cases}
\end{equation}
Here $\HT^{1}(\Rn)$ is the local Hardy space from \cite{Goldberg79}. Recall that $\HT^{1}(\Rn)$ consists of all $f\in\Sw'(\Rn)$ such that $q(D)f\in L^{1}(\Rn)$ and $(1-q)(D)f\in H^{1}(\Rn)$, for $H^{1}(\Rn)$ the classical Hardy space, endowed with the norm
\begin{equation}\label{eq:H1norm}
\|f\|_{\HT^{1}(\Rn)}:=\|q(D)f\|_{L^{1}(\Rn)}+\|(1-q)(D)f\|_{H^{1}(\Rn)}.
\end{equation}
Throughout, $q\in C^{\infty}_{c}(\Rn)$ is such that $q(\zeta)=1$ for $|\zeta|\leq 2$. Moreover, $\bmo(\Rn)$ consists of all $f\in\Sw'(\Rn)$ such that $q(D)f\in L^{\infty}(\Rn)$ and $(1-q)(D)f\in \BMO(\Rn)$, with norm
\begin{equation}\label{eq:bmonorm}
\|f\|_{\bmo(\Rn)}:=\|q(D)f\|_{L^{\infty}(\Rn)}+\|(1-q)(D)f\|_{\BMO(\Rn)}.
\end{equation}
Also recall that $\bmo(\Rn)$ is the dual of $\HT^{1}(\Rn)$.

Fix a non-negative radial $\ph\in C^{\infty}_{c}(\Rn)$ such that $\ph(\zeta)=0$ for $|\zeta|>1$, and $\ph\equiv1$ in a neighborhood of zero. For $\w\in S^{n-1}$, $\sigma>0$ and $\zeta\in\Rn\setminus\{0\}$, set $\ph_{\w,\sigma}(\zeta):=c_{\sigma}\ph\big(\tfrac{\hat{\zeta}-\w}{\sqrt{\sigma}}\big)$, where $c_{\sigma}:=\big(\int_{S^{n-1}}\ph\big(\tfrac{e_{1}-\nu}{\sqrt{\sigma}}\big)^{2}\ud\nu\big)^{-1/2}$ for $e_{1}=(1,0,\ldots,0)$ the first basis vector of $\Rn$ (this choice is irrelevant). Also let $\ph_{\w,\sigma}(0):=0$. Next, fix a non-negative radial $\Psi\in C^{\infty}_{c}(\Rn)$ such that $\Psi(\zeta)=0$ if $|\zeta|\notin[1/2,2]$, and
\[
\int_{0}^{\infty}\Psi(\sigma\zeta)^{2}\frac{\ud \sigma}{\sigma}=1\quad(\zeta\neq0).
\]
For $\w\in S^{n-1}$ and $\zeta\in\Rn$, set 
\[
\ph_{\w}(\zeta):=\int_{0}^{4}\Psi(\tau\zeta)\ph_{\w,\tau}(\zeta)\frac{\ud\tau}{\tau}.
\]
We will use the following properties of these functions, cf.~\cite[Remark 3.3]{Rozendaal21} and~\cite[Lemma 2.2]{FaLiRoSo23}:
\begin{enumerate}
\item For all $\w\in S^{n-1}$ and $\zeta\neq0$ one has 
\begin{equation}\label{eq:phiwsupport}
\ph_{\w}(\zeta)=0\text{ if }|\zeta|<\tfrac{1}{8}\text{ or }|\hat{\zeta}-\w|>2|\zeta|^{-1/2};
\end{equation}
\item For all $\alpha\in\Z_{+}^{n}$ and $\beta\in\Z_{+}$ there exists a $C_{\alpha,\beta}\geq0$ such that
\begin{equation}\label{eq:phiwgrowth}
|(\w\cdot \nabla_{\zeta})^{\beta}\partial^{\alpha}_{\zeta}\ph_{\w}(\zeta)|\leq C_{\alpha,\beta}|\zeta|^{\frac{n-1}{4}-\frac{|\alpha|}{2}-\beta}
\end{equation}
for all $\w\in S^{n-1}$ and $\zeta\neq0$;
\item\label{it:phiproperties3} For all $\alpha\in\Z_{+}^{n}$ there exists a $C_{\alpha}\geq0$ such that
\[
\Big|\partial_{\zeta}^{\alpha}\Big(\int_{S^{n-1}}\ph_{\w}(\zeta)\ud\w\Big)^{-1}\Big|\leq C_{\alpha} |\zeta|^{\frac{n-1}{4}-|\alpha|}
\]
for all $\zeta\in\Rn$ with $|\zeta|\geq1/2$. Hence there exists a radial $m\in S^{(n-1)/4}(\Rn)$ such that if $f\in\Sw'(\Rn)$ satisfies $\supp(\wh{f}\,)\subseteq \{\zeta\in\Rn\mid |\zeta|\geq1/2\}$, then 
\begin{equation}\label{eq:phiwintegral}
f=\int_{S^{n-1}}m(D)\ph_{\nu}(D)f\ud \nu.
\end{equation}
\end{enumerate}
In \eqref{it:phiproperties3} and throughout, for $\gamma\in\R$ we denote by $S^{\gamma}(\Rn)$ the class of $m\in C^{\infty}(\Rn)$ such that for each $\alpha\in\Z_{n}^{+}$ there exists a $C_{\alpha}'\geq0$ such that
\[
|\partial_{\zeta}^{\alpha}m(\zeta)|\leq C_{\alpha}'|\zeta|^{\gamma-|\alpha|}
\]
for all $\zeta\in\Rn$. Each $m\in S^{\gamma}(\Rn)$ yields a bounded operator $m(D):\HT^{s+\gamma,p}(\Rn)\to \HT^{s,p}(\Rn)$ for all $s\in\R$ (see Lemma \ref{lem:pseudoLp}).

Recall that $q\in C^{\infty}_{c}(\Rn)$ is such that $q(\zeta)=1$ for $|\zeta|\leq 2$. We can now define the Hardy spaces for Fourier integral operators.

\begin{definition}\label{def:HpFIO}
For $p\in[1,\infty)$ we let $\Hp$ consists of all $f\in\Sw'(\Rn)$ such that $q(D)f\in L^{p}(\Rn)$, $\ph_{\w}(D)f\in \HT^{p}(\Rn)$ for almost all $\w\in S^{n-1}$, and 
\[
\Big(\int_{S^{n-1}}\|\ph_{\w}(D)f\|_{\HT^{p}(\Rn)}^{p}\ud\w\Big)^{1/p}<\infty,
\]
endowed with the norm
\[
\|f\|_{\Hp}:=\|q(D)f\|_{L^{p}(\Rn)}+\Big(\int_{S^{n-1}}\|\ph_{\w}(D)f\|_{\HT^{p}(\Rn)}^{p}\ud\w\Big)^{1/p}.
\]
We also set $\HT^{\infty}_{FIO}(\Rn):=(\HT^{1}_{FIO}(\Rn))^{*}$. Moreover, for $p\in[1,\infty]$ and $s\in\R$ we let $\Hps$ consist of all $f\in\Sw'(\Rn)$ such that $\lb D\rb^{s}f\in\Hp$, endowed with the norm
\[
\|f\|_{\Hps}:=\|\lb D\rb^{s}f\|_{\Hp}.
\] 
\end{definition}

It is straightforward to see that, for $1\leq p<\infty$ and $s\in\R$, one has
\begin{equation}\label{eq:HpFIOnorm}
\|f\|_{\Hps}\eqsim \|q(D)f\|_{L^{p}(\Rn)}+\Big(\int_{S^{n-1}}\|\ph_{\w}(D)f\|_{\HT^{s,p}(\Rn)}^{p}\ud\w\Big)^{1/p}
\end{equation}
for all $f\in\Hps$, with implicit constants independent of $f$.

\begin{remark}\label{rem:defHpFIO}
We note that this is not how the Hardy spaces for Fourier integral operators were originally defined in~\cite{Smith98a} and~\cite{HaPoRo20}. There these spaces were introduced using a conical square function over the cosphere bundle, or equivalently using wave packet transforms and tent spaces over the cosphere bundle. That definition also includes an intrinsic description of $\HT^{\infty}_{FIO}(\Rn)$. It was shown in \cite{Rozendaal21} for $1<p<\infty$, and in \cite{FaLiRoSo23} for $p=1$, that the spaces in Definition \ref{def:HpFIO} coincide with the original Hardy spaces for Fourier integral operators. The same then holds for $p=\infty$, by duality (see \cite[Proposition 6.8]{HaPoRo20}).
\end{remark}

For $p\in[1,\infty]$ set
\begin{equation}\label{eq:sp}
s(p):=\tfrac{n-1}{2}|\tfrac{1}{p}-\tfrac{1}{2}|.
\end{equation}
With notation as in \eqref{eq:classical}, by~\cite[Theorem 7.4]{HaPoRo20} the following Sobolev embeddings hold for all $p\in[1,\infty]$ and $s\in\R$:
\begin{equation}\label{eq:sobolev}
\HT^{s+s(p),p}(\Rn)\subseteq\Hps\subseteq\HT^{s-s(p),p}(\Rn),
\end{equation}
and the exponents in these embeddings are sharp (see~\cite[Remark 7.9]{HaPoRo20} and~\cite[Remark 6.5]{FaLiRoSo23}).

A key idea in this article is to apply Littlewood-Paley theory to \eqref{eq:HpFIOnorm}. Throughout, $(\psi_{j})_{j=0}^{\infty}\subseteq C^{\infty}_{c}(\Rn)$ is a fixed Littlewood--Paley decomposition. That is,
\begin{equation}\label{eq:LitPaldecomp}
\sum_{j=0}^{\infty}\psi_{j}(\xi)=1\quad(\xi\in\Rn),
\end{equation}
$\psi_{0}(\xi)=0$ for $|\xi|>1$, $\psi_{1}(\xi)=0$ if $|\xi|\notin [1/2,2]$, and $\psi_{j}(\xi)=\psi_{1}(2^{-j+1}\xi)$ for all $j>1$ and $\xi\in\Rn$. In fact, one can assume that $\psi_{1}(\xi)=0$ for $|\xi|\notin [(1+\veps)/2,2-\veps]$ for some $\veps>0$, an additional assumption which is of minor convenience in the proof of Theorem \ref{thm:critical} below. For notational simplicity, we also set $\psi_{j}:=0$ for $j<0$. We will use the following standard lemma (see e.g.~\cite[Section 2.5]{Triebel10}).

\begin{lemma}\label{lem:LitPal}
Let $p\in[1,\infty)$ and $s\in\R$. Let $(\chi_{k})_{k=0}^{\infty}\subseteq C^{\infty}_{c}(\Rn)$ be such that $\chi_{k}(\xi)=\chi_{1}(2^{-k+1}\xi)$ for all $k>1$ and $\xi\in\Rn$, and $\chi_{1}(\xi)=0$ if $|\xi|\notin [1/4,4]$. Then there exists a $C>0$ such that the following holds for all $(f_{k})_{k=0}^{\infty}\subseteq L^{p}(\Rn)$ with
\[
\supp(\wh{f}_{k})\subseteq\{\xi\in\Rn\mid |\xi|\in[2^{k-3},2^{k+1}]\}
\]
for $k\geq1$, and $\supp(\wh{f}_{0})\subseteq\{\xi\in\Rn\mid |\xi|\leq 2\}$.
\begin{enumerate}
\item\label{it:LitPal1} 
If $(\sum_{k=0}^{\infty}4^{ks}|f_{k}|^{2})^{1/2}\in L^{p}(\Rn)$ then $\sum_{k=0}^{\infty}f_{k}\in \HT^{s,p}(\Rn)$ and
\[
\Big\|\sum_{k=0}^{\infty}f_{k}\Big\|_{\HT^{s,p}(\Rn)}\leq C\Big\|\Big(\sum_{k=0}^{\infty}4^{ks}|f_{k}|^{2}\Big)^{1/2}\Big\|_{L^{p}(\Rn)}.
\]
\item\label{it:LitPal2}
If $f_{k}=\chi_{k}(D)f$ for all $k\geq0$ and some $f\in \Sw'(\Rn)$, and if $\sum_{k=0}^{\infty}f_{k}\in \HT^{s,p}(\Rn)$, then $(\sum_{k=0}^{\infty}4^{ks}|f_{k}|^{2})^{1/2}\in L^{p}(\Rn)$ and
\[
\Big\|\Big(\sum_{k=0}^{\infty}4^{ks}|f_{k}|^{2}\Big)^{1/2}\Big\|_{L^{p}(\Rn)}\leq C\Big\|\sum_{k=0}^{\infty}f_{k}\Big\|_{\HT^{s,p}(\Rn)}.
\]
\end{enumerate}
\end{lemma}

We will combine Lemma \ref{lem:LitPal} with the following result about multipliers on the space $L^{p}(\Rn;\ell^{2})$ of $p$-integrable $\ell^{2}$-valued functions. Recall that $e_{1}$ is the first basis vector of $\Rn$, although this choice is again immaterial. 

\begin{lemma}\label{lem:multiplier}
Let $p\in(1,\infty)$. Then $\lb D\rb^{-\frac{n-1}{4}}\ph_{\w}(D)\in\La(L^{p}(\Rn;\ell^{2}))$ for all $\w\in S^{n-1}$, and
\[
\|\lb D\rb^{-\frac{n-1}{4}}\ph_{\w}(D)\|_{\La(L^{p}(\Rn;\ell^{2}))}=\|\lb D\rb^{-\frac{n-1}{4}}\ph_{e_{1}}(D)\|_{\La(L^{p}(\Rn;\ell^{2}))}.
\]
\end{lemma}
\begin{proof}
By applying a rotation, using the definition of $\ph_{\w}$ and the assumption that $\ph$ is radial, it suffices to show that $\lb D\rb^{-\frac{n-1}{4}}\ph_{e_{1}}(D)\in \La(L^{p}(\Rn;\ell^{2}))$. To this end we claim that, for every $\alpha\in\Z_{+}^{n}$, one has
\begin{equation}\label{eq:anisotropic}
\sup_{\xi\in\Rn}\big|\xi^{\alpha}\partial_{\xi}^{\alpha}\big(\lb\xi\rb^{-\frac{n-1}{4}}\ph_{e_{1}}(\xi)\big)\big|<\infty.
\end{equation}
Taking the claim for granted momentarily, it then follows from the Marcinkiewicz--Lizorkin multiplier theorem (see e.g.~\cite[Proposition 3]{Zimmermann89} for a version with values in a Banach space) that $\lb D\rb^{-\frac{n-1}{4}}\ph_{e_{1}}(D)\in \La(L^{p}(\Rn;\ell^{2}))$.

We conclude by proving \eqref{eq:anisotropic}. Let $\alpha=(\alpha_{1},\ldots,\alpha_{n})\in\Z_{+}^{n}$ and $\xi=(\xi_{1},\ldots,\xi_{n})\in\supp(\ph_{e_{1}})$, and write $\xi':=\xi-\xi_{1}e_{1}=(0,\xi_{2},\ldots,\xi_{n})$. Then, by \eqref{eq:phiwsupport}, one has
\[
|\xi'|^{2}\leq (\xi_{1}-|\xi|)^{2}+\xi_{2}^{2}+\ldots+\xi_{n}^{2}=|\xi-|\xi|e_{1}|^{2}=|\xi|^{2}|\hat{\xi}-e_{1}|\lesssim |\xi|.
\]
Hence, with $\alpha':=(0,\alpha_{2},\ldots,\alpha_{n})$, \eqref{eq:phiwgrowth} yields
\begin{align*}
|\xi^{\alpha}\partial_{\xi}^{\alpha}\ph_{e_{1}}(\xi)|=|\xi_{1}^{\alpha_{1}}(\xi')^{\alpha'}\partial_{\xi_{1}}^{\alpha_{1}}\partial_{\xi}^{\alpha'}\ph_{e_{1}}(\xi)|\lesssim |\xi|^{\alpha_{1}+\frac{\alpha'}{2}}|\xi|^{\frac{n-1}{4}-\alpha_{1}-\frac{\alpha'}{2}}=|\xi|^{\frac{n-1}{4}}.
\end{align*}
Now \eqref{eq:anisotropic} follows from the Leibniz rule.
\end{proof}

\begin{remark}\label{rem:CalZyg}
As was pointed out by the referee, $\lb D\rb^{-(n-1)/4}\ph_{e_{1}}(D)$ is in fact an anisotropic Calderon--Zygmund operator, associated to the dilation $\xi\mapsto (r^{2}\xi_{1},r\xi')$ for $r>1$. Hence it also follows from the general theory of anisotropic Calderon--Zygmund operators that $\lb D\rb^{-\frac{n-1}{4}}\ph_{e_{1}}(D)\in \La(L^{p}(\Rn;\ell^{2}))$. See, for example, \cite{Bownik03} in the scalar-valued case, with the extension to functions with values in a Hilbert space being standard (see \cite[Chapter I]{Stein93}).
\end{remark}

\subsection{Rough symbols}\label{subsec:symbols}

We first introduce some classes of rough symbols, and then we describe a smoothing procedure for such symbols. 

Let $(\psi_{j})_{j=0}^{\infty}\subseteq C^{\infty}_{c}(\Rn)$ be the Littlewood--Paley decomposition from \eqref{eq:LitPaldecomp}. For $r>0$ we let the \emph{Zygmund space} $C^{r}_{*}(\Rn)$ consist of all $f\in\Sw'(\Rn)$ such that $\psi_{j}(D)f\in L^{\infty}(\Rn)$ for all $j\geq0$, with
\[
\|f\|_{C^{r}_{*}(\Rn)}:=\sup_{j\in\Z_{+}}2^{jr}\|\psi_{j}(D)f\|_{L^{\infty}(\Rn)}<\infty.
\]
Then (see~\cite{Triebel10})
\begin{equation}\label{eq:Zygmund1}
\Hrinf(\Rn)\subsetneq C^{r}_{*}(\Rn)=C^{r}(\Rn)\cap L^{\infty}(\Rn)
\end{equation}
if $r\notin\N$, and
\begin{equation}\label{eq:Zygmund2}
C^{r-1,1}(\Rn)\cap L^{\infty}(\Rn)\subsetneq \Hrinf(\Rn)\subsetneq C^{r}_{*}(\Rn)
\end{equation}
if $r\in\N$. Here $\Hrinf(\Rn)$ is as in \eqref{eq:classical}. Moreover, for $r=l+s\notin\N$ with $l\in\Z_{+}$ and $s\in(0,1)$, $C^{r}(\Rn)$ consists of those $f\in C^{l}(\Rn)$ such that for each $\alpha\in\Z_{+}^{n}$ with $|\alpha|=l$, the partial derivative $\partial_{x}^{\alpha}f$ is H\"{o}lder continuous with parameter $s$. For $r\in \N$, $C^{r-1,1}(\Rn)$ consists of those $f\in C^{r-1}(\Rn)$ such that $\partial_{x}^{\alpha}f$ is Lipschitz for all $\alpha\in\Z_{+}^{n}$ with $|\alpha|=r-1$. We also note in passing that $C^{r}_{*}(\Rn)$ is equal to the Besov space $B^{r}_{\infty,\infty}(\Rn)$ for all $r>0$.

Recall that, for $m\in\R$ and $\delta\in[0,1]$, the H\"{o}rmander symbol class $S^{m}_{1,\delta}$ is the space of $a\in C^{\infty}(\R^{2n})$ such that, for all $\alpha,\beta\in\Z_{+}^{n}$, there exists a $C_{\alpha,\beta}\geq0$ with
\[
|\partial_{x}^{\beta}\partial_{\eta}^{\alpha}a(x,\eta)|\leq C_{\alpha,\beta}\lb\eta\rb^{m-|\alpha|+\delta|\beta|}
\]
for all $x,\eta\in\Rn$. We will be interested in versions of these symbols that are rough in the $x$ variable. 

\begin{definition}\label{def:rough}
Let $r>0$, $m\in\R$ and $\delta\in[0,1]$. Then $C^{r}_{*}S^{m}_{1,\delta}$ is the collection of $a:\R^{2n}\to\C$ such that for each $\alpha\in\Z_{+}^{n}$ there exists a $C_{\alpha}\geq 0$ with the following properties:
\begin{enumerate}
\item\label{it:symbol1} For all $x,\eta\in\Rn$ one has $a(x,\cdot)\in C^{\infty}(\Rn)$ and
\begin{equation}\label{eq:symbol1}
|\partial_{\eta}^{\alpha}a(x,\eta)|\leq C_{\alpha}\lb\eta\rb^{m-|\alpha|}.
\end{equation}
\item\label{it:symbol2} For all $\eta\in\Rn$ one has $\partial_{\eta}^{\alpha}a(\cdot,\eta)\in C^{r}_{*}(\Rn)$ and
\begin{equation}\label{eq:symbol2}
\|\partial_{\eta}^{\alpha}a(\cdot,\eta)\|_{C^{r}_{*}(\Rn)}\leq C_{\alpha}\lb\eta\rb^{m-|\alpha|+r\delta}.
\end{equation}
\end{enumerate} 
Moreover, $\Hrinf S^{m}_{1,\delta}$ is the collection of $a\in C^{r}_{*}S^{m}_{1,\delta}$ such that, in \eqref{it:symbol2}, for all $\eta\in\Rn$ one has $\partial_{\eta}^{\alpha}a(\cdot,\eta)\in \Hrinf(\Rn)$ and 
\[
\|\partial_{\eta}^{\alpha}a(\cdot,\eta)\|_{\Hrinf(\Rn)}\leq C_{\alpha}\lb\eta\rb^{m-|\alpha|+r\delta}.
\]
\end{definition}

Note that $S^{m}_{1,\delta}\subsetneq \Hrinf S^{m}_{1,\delta}\subsetneq C^{r}_{*}S^{m}_{1,\delta}$ for all $r>0$, $m\in\R$ and $\delta\in[0,1]$. 

Given $a\in C^{r}_{*}S^{m}_{1,\delta}$ for some $r>0$, $m\in\R$ and $\delta\in[0,1]$, we define the pseudodifferential operator $a(x,D):\Sw(\Rn)\to\Sw'(\Rn)$ with symbol $a$ by
\begin{equation}\label{eq:pseudodef}
a(x,D)f(x):=\frac{1}{(2\pi)^{n}}\int_{\Rn}e^{ix\eta}a(x,\eta)\wh{f}(\eta)\ud\eta
\end{equation}
for $f\in\Sw(\Rn)$ and $x\in\Rn$.

\begin{remark}\label{rem:dualitydef}
In \eqref{eq:pseudodef} we only defined $a(x,D)f$ for $f\in\Sw(\Rn)$, whereas we will consider the boundedness of $a(x,D)$ on $\HT^{s,p}(\Rn)$ and $\Hps$, for $p\in[1,\infty]$ and $s\in\R$. For $p<\infty$ the Schwartz functions lie dense in $\HT^{s,p}(\Rn)$ and $\Hps$ for all $s\in\R$ (see \cite[Proposition 6.6]{HaPoRo20}), so by obtaining suitable norm bounds for $f\in\Sw(\Rn)$ we can consider the unique extension of $a(x,D)$ to these spaces. However, for $p=\infty$ the Schwartz functions are not dense in $\HT^{s,\infty}(\Rn)$ and $\HT^{s,\infty}_{FIO}(\Rn)$, and extending \eqref{eq:pseudodef} by adjoint action to $f\in\Sw'(\Rn)$ is problematic both due to the low regularity of $a$ and, even in the smooth case, for $\delta=1$ and $a\in S^{m}_{1,1}$. For suitable $s\in\R$ one can define $a(x,D)f$ for $f\in \HT^{s,\infty}(\Rn)$ or $f\in\HT^{s,\infty}_{FIO}(\Rn)$ by duality, using also a symbol decomposition as in the proof of Theorem \ref{thm:critical} (see~\cite{Marschall88} or~\cite[Section 1.2]{Taylor00}). However, because the case $p=\infty$ plays a very minor role in this article, we will not concern ourselves with such matters. Instead, when we write $a(x,D):X\to Y$ for function spaces $X$ and $Y$ over $\Rn$ that contain $\Sw(\Rn)$, we simply mean that $\|a(x,D)f\|_{Y}\lesssim \|f\|_{X}$ for all $f\in\Sw(\Rn)$.
\end{remark}

 A basic pseudodifferential operator is the multiplication operator $a(x,D)$ associated with a function $a\in\Hrinf(\Rn)\subseteq \Hrinf S^{0}_{1,0}$. Such multiplication operators in turn give rise to other pseudodifferential operators, through paradifferential calculus. For example, for $r>0$ and $b\in \Hrinf(\Rn)$, we define the high-high paraproduct operator $R_{b}$ and the high-low paraproduct operator $\pi_{b}$ as follows. For $f\in\Sw(\Rn)$ and $x\in\Rn$ set
\begin{equation}\label{eq:hihi}
R_{b}(f)(x):=\sum_{k=0}^{\infty}\sum_{j=k-5}^{k+5}(\psi_{j}(D)b)(x)(\psi_{k}(D)f)(x)
\end{equation}
and
\begin{equation}\label{eq:hilo}
\pi_{b}(f)(x):=\sum_{k=0}^{\infty}\sum_{j=k+6}^{\infty}(\psi_{j}(D)b)(x)(\psi_{k}(D)f)(x).
\end{equation}
One can check (see also the proof of Lemma \ref{lem:smoothing}) that $R_{b}$ and $\pi_{b}$ are pseudodifferential operators with symbols in $S^{-r}_{1,1}$ and $\Hrinf S^{-r}_{1,1}$, respectively.  
A minor role will be played in the proof of Lemma \ref{lem:pseudoLp} by the following mapping property of these operators, the proof of which can be skipping during a first reading.

\begin{lemma}\label{lem:paraproducts}
Let $r>0$ and $b\in\Hrinf(\Rn)$. Then
\[
R_{b},\pi_{b}:\HT^{-r,p}(\Rn)\to\HT^{-\veps,p}(\Rn)
\]
for all $p\in[1,\infty]$ and $\veps>0$.
\end{lemma}
\begin{proof}
By decomposing $b$ into a sum of finitely many functions with restricted Fourier supports and then moving derivatives around as in the proofs of~\cite[Proposition 3.5.B and Proposition 3.5.F]{Taylor91}, it suffices to show that for each $c\in\HT^{0,\infty}(\Rn)=\bmo(\Rn)$ one has
\begin{equation}\label{eq:parareduced}
R_{c},\pi_{c}:\HT^{0,p}(\Rn)\to\HT^{-\veps,p}(\Rn).
\end{equation}
Since $R_{c}$ and $\pi_{c}$ are Calderon-Zygmund operators, as follows for example from~\cite[Lemma 3.5.E]{Taylor91}, one then immediately obtains \eqref{eq:parareduced} for $1<p<\infty$. 

To deal with $p=1$ and $p=\infty$, we argue in different ways for $R_{c}$ and $\pi_{c}$. We first consider $\pi_{c}$, and we give an argument which in fact directly works for all $p\in[1,\infty]$. Let $f\in \Sw(\Rn)$ and use the support properties of $\psi_{k}$, as well as the fact that $(\F^{-1}(\psi_{k}))_{k=0}^{\infty}\subseteq L^{1}(\Rn)$ is uniformly bounded, to find an $M\in\Z_{+}$ such that
\begin{align*}
&2^{-l\veps/2}\|\psi_{l}(D)\pi_{c}(f)\|_{L^{p}}\leq 2^{-l\veps/2}\sum_{k=0}^{\infty}\sum_{j=k+6}^{\infty}\big\|\psi_{l}(D)\big((\psi_{j}(D)c)(\psi_{k}(D)f)\big)\big\|_{L^{p}}\\
&=2^{-l\veps/2}\sum_{j=l-M}^{l+M}\sum_{k=0}^{j-6}\big\|\psi_{l}(D)\big((\psi_{j}(D)c)(\psi_{k}(D)f)\big)\big\|_{L^{p}}\\
&\lesssim 2^{-l\veps/2}\sum_{j=l-M}^{l+M}\sum_{k=0}^{j-6}\|(\psi_{j}(D)c)(\psi_{k}(D)f)\|_{L^{p}}\\
&\lesssim \sum_{j=l-M}^{l+M}2^{-j\veps/2}\sum_{k=0}^{j-6}\|\psi_{j}(D)c\|_{L^{\infty}}\|\psi_{k}(D)f\|_{L^{p}}\lesssim j2^{-j\veps/2}\|c\|_{\bmo}\sup_{k\geq0}\|\psi_{k}(D)f\|_{L^{p}}
\end{align*}
for all $l\geq0$. For the final inequality we used that $\|\psi_{j}(D)c\|_{L^{\infty}}\eqsim \|\psi_{j}(D)c\|_{\bmo}\lesssim \|c\|_{\bmo}$, a standard fact which follows for example from the Littlewood-Paley decomposition of $H^{1}(\Rn)$ and duality, combined with the uniform boundedness of $(\F^{-1}(\psi_{k}))_{k=0}^{\infty}$ in $L^{1}(\Rn)$. Now the Besov space embeddings
\begin{equation}\label{eq:Besovemb}
\HT^{0,p}(\Rn)\subseteq B^{0}_{p,\infty}(\Rn)\subseteq B^{-\veps}_{p,1}(\Rn)\subseteq \HT^{-\veps,p}(\Rn)
\end{equation}
from~\cite[Sections 2.3.2, 2.5.7 and 2.5.8]{Triebel10} conclude the proof for $\pi_{c}$.

Next, for $R_{c}$ we note that, since $R_{c}$ is a Calderon-Zygmund operator,~\cite[Theorem 4.2.6]{Grafakos14b} yields $R_{c}:H^{1}(\Rn)\to L^{1}(\Rn)$. Combined with the embedding $L^{1}(\Rn)\subseteq\HT^{-\veps,1}(\Rn)$, which in turn follows from \eqref{eq:Besovemb} because $L^{1}(\Rn)\subseteq B^{0}_{1,\infty}(\Rn)$, we can use \eqref{eq:H1norm} to see that 
\[
R_{c}(1-q)(D):\HT^{1}(\Rn)\to L^{1}(\Rn)\subseteq \HT^{-\veps,1}(\Rn).
\]
Moreover, since $q\in C^{\infty}_{c}(\Rn)$, there exists an $N\geq0$ such that
\begin{align*}
\|R_{c}(q(D)f)\|_{L^{1}}&\leq \sum_{k=0}^{N}\sum_{j=k-5}^{k+5}\|(\psi_{j}(D)c)(\psi_{k}(D)q(D)f)\|_{L^{1}}\\
&\lesssim \sum_{k=0}^{N}\sum_{j=k-5}^{k+5}\|\psi_{j}(D)c\|_{L^{\infty}}\|\psi_{k}(D)q(D)f\|_{L^{1}}\\
&\lesssim \sum_{k=0}^{N}\sum_{j=k-5}^{k+5}\|c\|_{\bmo}\|q(D)f\|_{L^{1}}\lesssim \|c\|_{\bmo}\|f\|_{\HT^{1}}
\end{align*}
for all $f\in \Sw(\Rn)$. This proves \eqref{eq:parareduced} for $R_{c}$ and $p=1$.

Finally, for $p=\infty$ one can use~\cite[Lemma 3.3, Definition 3.11 and Corollary 7.3]{Rodriguez-Staubach15}, since $(c,f)\mapsto R_{c}(f)$ is a bilinear Coifman-Meyer multiplier.
\end{proof}

We next describe a symbol smoothing procedure that decomposes a rough symbol as a sum of a smooth part and a rough part with additional decay. Let $a\in C^{r}_{*}S^{m}_{1,\delta}$ for $r>0$, $m\in\R$ and $\delta\in[0,1]$. Let $\gamma\in[0,1]$ be given, and recall from Section \ref{subsec:Hardy} that $\ph\in C^{\infty}_{c}(\Rn)$ satisfies $\ph\equiv1$ near zero. Now set, for $(x,\eta)\in\R^{2n}$,
\[
a^{\sharp}_{\gamma}(x,\eta):=\sum_{k=0}^{\infty}\big(\ph(2^{-\gamma k}D)a(\cdot,\eta)\big)(x)\psi_{k}(\eta)
\]
and
\[
a^{\flat}_{\gamma}(x,\eta):=a(x,\eta)-a^{\sharp}_{\gamma}(x,\eta)=\sum_{k=0}^{\infty}\big((1-\ph)(2^{-\gamma k}D)a(\cdot,\eta)\big)(x)\psi_{k}(\eta).
\]
For the final identity we used that $\sum_{k=0}^{\infty}\psi_{k}(\eta)=1$. The decomposition $a=a^{\sharp}_{\gamma}+a^{\flat}_{\gamma}$ has the following properties.

\begin{lemma}\label{lem:smoothing}
Let $r>0$, $m\in\R$ and $\delta,\gamma\in[0,1]$ with $\gamma\geq\delta$. Then for each $a\in C^{r}_{*}S^{m}_{1,\delta}$ one has $a^{\sharp}_{\gamma}\in S^{m}_{1,\gamma}$ and $a^{\flat}_{\gamma}\in C^{r}_{*}S^{m-(\gamma-\delta)r}_{1,\gamma}$. If $a\in \Hrinf S^{m}_{1,\delta}$ then additionally $a^{\flat}_{\gamma}\in \Hrinf S^{m-(\gamma-\delta)r}_{1,\gamma}$.
\end{lemma}
\begin{proof}
The result is essentially contained in~\cite[Section 1.3]{Taylor91} (see also~\cite[Section 1.3]{Taylor00}), but for the convenience of the reader we sketch the argument here. Let $\alpha,\beta\in\Z_{+}^{n}$ and $\eta\in\Rn$. Since $\ph\in C^{\infty}_{c}(\Rn)$, we can use an $L^{1}$-estimate for $\partial_{x}^{\beta}\F^{-1}(\ph)$, as well as \eqref{eq:symbol1}, to obtain
\[
\big\|\partial_{x}^{\beta}\partial_{\eta}^{\alpha}\big(\ph(2^{-\gamma k}D)a(\cdot,\eta)\big)\big\|_{L^{\infty}}\lesssim 2^{k\gamma |\beta|}\|\partial_{\eta}^{\alpha}a(\cdot,\eta)\|_{L^{\infty}}\lesssim 2^{k\gamma|\beta|}\lb\eta\rb^{m-|\alpha|}
\]
for all $k\geq0$. Noting that $\lb\eta\rb\eqsim 2^{k}$ if $\eta\in\supp(\psi_{k})$, one can derive from this that $a^{\sharp}_{\gamma}\in S^{m}_{1,\gamma}$. 

On the other hand, for all $\xi\in\Rn$ and $j\geq0$ one has $\psi_{j}(\xi)=0$ if $|\xi|\notin [2^{j-2},2^{j}]$, and there exists a $c>0$ such that $\ph(\xi)=1$ if $|\xi|\leq c$. Hence, for some $M\geq0$, \eqref{eq:symbol2} yields
\begin{align*}
&\|\partial_{\eta}^{\alpha}(1-\ph)(2^{-\gamma k}D)a(\cdot,\eta)\|_{L^{\infty}}\leq \sum_{j\geq \gamma k-M}\|(1-\ph)(2^{-\gamma k}D)\psi_{j}(D)\partial_{\eta}^{\alpha}a(\cdot,\eta)\|_{L^{\infty}}\\
&\lesssim \sum_{j\geq\gamma k-M}\|\psi_{j}(D)\partial_{\eta}^{\alpha}a(\cdot,\eta)\|_{L^{\infty}}\leq \sum_{j\geq\gamma k-M}2^{-jr}\|\partial_{\eta}^{\alpha}a(\cdot,\eta)\|_{C^{r}_{*}}\lesssim 2^{-kr\gamma }\lb\eta\rb^{m-|\alpha|+r\delta}
\end{align*}
for all $k\geq0$. In the same manner as before, this implies that $\|\partial^{\alpha}_{\eta}a^{\flat}_{\gamma}(\cdot,\eta)\|_{L^{\infty}}\lesssim \lb\eta\rb^{m-(\gamma-\delta)r-|\alpha|}$. Finally, by \eqref{eq:symbol2} and because $C^{r}_{*}(\Rn)$ is preserved under convolution with elements of $L^{1}(\Rn)$, one has
\begin{equation}\label{eq:aflatbound}
\|\partial_{\eta}^{\alpha}(1-\ph)(2^{-\gamma k}D)a(\cdot,\eta)\|_{C^{r}_{*}}\lesssim \|\partial_{\eta}^{\alpha}a(\cdot,\eta)\|_{C^{r}_{*}}\lesssim 2^{m-|\alpha|+\delta r},
\end{equation} 
which in turn implies that $a^{\flat}_{\gamma}\in C^{r}_{*}S^{m-(\gamma-\delta)r}_{1,\gamma}$. If $a\in \Hrinf S^{m}_{1,\delta}$ then \eqref{eq:aflatbound} holds with $C^{r}_{*}$ replaced by $\Hrinf $, so that $a^{\flat}_{\gamma}\in \Hrinf S^{m-(\gamma-\delta)r}_{1,\gamma}$.
\end{proof}

\section{The subcritical case}\label{sec:subcritical}

In this section we prove results about rough pseudodifferential operators acting on Hardy spaces for Fourier integral operators, in the subcritical case. The results in this section suffice for many purposes, except that they do not yield boundedness on $\Hps$ for a pseudodifferential operator with symbol $a\in C^{r}_{*}S^{0}_{1,1/2}$ unless $p=2$. 

We first state an extension of a result from~\cite{Bourdaud82} about rough pseudodifferential operators acting on the classical function spaces $\HT^{s,p}(\Rn)$ from \eqref{eq:classical}. 

\begin{lemma}\label{lem:pseudoLp}
Let $r>0$, $m\in\R$, $\delta\in[0,1]$, $p\in[1,\infty]$ and $a\in C^{r}_{*}S^{m}_{1,\delta}$. Then 
\begin{equation}\label{eq:pseudoLp}
a(x,D):\HT^{s+m,p}(\Rn)\to\HT^{s,p}(\Rn)
\end{equation}
for $-(1-\delta)r<s<r$. If $\delta<1$ and $a\in \Hrinf S^{m}_{1,\delta}$, then \eqref{eq:pseudoLp} also holds for $s=r$. If additionally $a=b^{\flat}_{\delta}$ for some $b\in\Hrinf (\Rn)$, then \eqref{eq:pseudoLp} also holds for $s=-(1-\delta)r$ and $m=-\delta r$.
\end{lemma}
\begin{proof}
For the first statement see~\cite[Theorems 2.3 and 2.4]{Marschall88}. The second statement, concerning $s=r$ in \eqref{eq:pseudoLp}, is contained in~\cite[Theorem 2.2]{Marschall88}. For the final statement, let $R_{b}$ and $\pi_{b}$ be as in \eqref{eq:hihi} and \eqref{eq:hilo}, respectively, and write $b^{\flat}_{\delta}(x,D)=R_{b}+\pi_{b}+b^{\flat}_{\delta}(x,D)-R_{b}-\pi_{b}$. Lemma \ref{lem:paraproducts} shows that
\[
R_{b}+\pi_{b}:\HT^{-r,p}(\Rn)\to\HT^{-(1-\delta)r,p}(\Rn).
\]
Moreover, the symbol $c$ of $b^{\flat}_{\delta}(x,D)-R_{b}-\pi_{b}$ is given by $c=c_{1}+c_{2}$, where 
\[
c_{1}(x,\eta):=\sum_{k=1}^{\infty}\Big(-\ph(2^{-\delta k}D)+\sum_{j=0}^{k-6}\psi_{j}(D)\Big)b(\cdot,\eta)(x)\psi_{k}(\eta)
\] 
and
\[
c_{2}(x,\eta):=-\ph(D)b(\cdot,\eta)(x)\psi_{0}(\eta)
\]
for $x,\eta\in\Rn$. For $\ph$ supported sufficiently close to zero, one has $\F c_{1}(\cdot,\eta)(\xi)=0$ if $|\xi|\geq |\eta|/10$. Moreover, as in the proof of Lemma \ref{lem:smoothing} one can show that for all $\alpha,\beta\in\Z_{+}^{n}$ there exists a $C_{\alpha,\beta}\geq0$ such that
\[
|\partial_{x}^{\beta}\partial_{\xi}^{\alpha}c_{1}(x,\eta)|\leq C_{\alpha,\beta} \lb \eta\rb^{|\beta|-|\alpha|}
\]
for all $x,\eta\in\Rn$. Now~\cite[Proposition 3.4.F]{Taylor91} yields
\[
c_{1}(x,D):\HT^{t-\delta r,p}(\Rn)\to\HT^{t,p}(\Rn)
\]
for all $t\in\R$, and specifically for $t=s$. One can check directly that the same mapping property holds for $c_{2}(x,D)$. This concludes the proof.
\end{proof}

We next state a result about pseudodifferential operators with smooth symbols on $\Hps$. This lemma reduces to a theorem in~\cite{HaPoRo20}, but we note that it can also be proved in a similar manner as Theorem \ref{thm:critical} below (see Remark \ref{rem:smooth}).

\begin{lemma}\label{lem:smoothpseudo}
Let $m\in\R$ and $a\in S^{m}_{1,1/2}$. Then 
\[
a(x,D):\HT^{s+m,p}_{FIO}(\Rn)\to\Hps
\]
for all $p\in[1,\infty]$ and $s\in\R$.
\end{lemma}
\begin{proof}
In the case where $m=s=0$, the statement is contained in~\cite[Theorem 6.10]{HaPoRo20}. For general $m$ and $s$, it suffices to show that 
\[
\lb D\rb^{s}a(x,D)\lb D\rb^{-m-s}:\Hp\to\Hp
\]
is bounded. But $\lb D\rb^{s}a(x,D)\lb D\rb^{-m-s}$ has a symbol in $S^{0}_{1,1/2}$, cf.~\cite[Proposition 0.3.C]{Taylor91}, so the general case is reduced to what we have already shown.
\end{proof}

We now prove a proposition that yields, in particular, boundedness of pseudodifferential operators with $C^{r}_{*}S^{-\veps}_{1,\delta}$ symbols on $\Hps$, for any $r,\veps>0$ and $p$ close enough to $2$. Recall from \eqref{eq:sp} that, for $p\in[1,\infty]$, $s(p)=\frac{n-1}{2}|\frac{1}{2}-\frac{1}{p}|$.

\begin{proposition}\label{prop:pseudoloss}
Let $r>0$, $m\in\R$, $\delta\in[0,1]$, $p\in[1,\infty]$ and $a\in C^{r}_{*}S^{m}_{1,\delta}$. Then 
\begin{equation}\label{eq:pseudoHpFIO}
a(x,D):\HT^{s+2s(p)+m,p}_{FIO}(\Rn)\to\Hps
\end{equation}
for $-(1-\delta)r<s+s(p)<r$. If $\delta<1$ and $a\in \Hrinf S^{m}_{1,\delta}$, then \eqref{eq:pseudoHpFIO} also holds for $s+s(p)=r$. 
If additionally $a=b^{\flat}_{\delta}$ for some $b\in\Hrinf (\Rn)$, then \eqref{eq:pseudoHpFIO} also holds for $s+s(p)=-(1-\delta)r$ and $m=-\delta r$.
\end{proposition}
\begin{proof}
The first statement follows by combining Lemma \ref{lem:pseudoLp} and the Sobolev embeddings for $\Hps$ from \eqref{eq:sobolev}:
\[
a(x,D):\HT^{s+2s(p)+m,p}_{FIO}(\Rn)\subseteq \HT^{s+s(p)+m,p}(\Rn)\to \HT^{s+s(p),p}(\Rn)\subseteq \Hps.
\]
The remaining statements follow in the same manner from Lemma \ref{lem:pseudoLp} and \eqref{eq:sobolev}.
\end{proof}

\begin{remark}\label{rem:doubleflat}
For the proof of Theorem \ref{thm:critical} it is convenient to note that \eqref{eq:pseudoHpFIO} also holds for $s+s(p)=-(1-\delta)r$ and $m=-\delta r$ when $a=((b^{\flat}_{\delta'})^{\flat}_{\delta'})^{\flat}_{\delta}$ for some $b\in\Hrinf (\Rn)$ if $\delta'\in[0,\delta]$. This follows either from a minor modification of the proof of Lemma \ref{lem:pseudoLp}, or by noting that 
\[
((b^{\flat}_{\delta'})^{\flat}_{\delta'})^{\flat}_{\delta}(x,D)-b^{\flat}_{\delta}(x,D):\HT^{t-\delta r,p}(\Rn)\to\HT^{t,p}(\Rn)
\]
is bounded for all $t\in\R$, as in the proof of Lemma \ref{lem:pseudoLp}.
\end{remark}

By combining Proposition \ref{prop:pseudoloss} with the symbol smoothing procedure from Section \ref{subsec:symbols}, we obtain the following corollary. Much of this result will in fact be improved in Corollary \ref{cor:combined}. However, the range of Sobolev indices in the latter result is smaller than what we will obtain here. Moreover, the proof of Corollary \ref{cor:combined}, via Theorem \ref{thm:critical}, is relatively involved, and we believe that there is value in demonstrating how a simpler proof yields a boundedness statement that already suffices for many purposes.

\begin{corollary}\label{cor:subcritical}
Let $r>0$, $m\in\R$, $\delta\in[0,1/2]$, $p\in[1,\infty]$ and $a\in C^{r}_{*}S^{m}_{1,\delta}$. Set $\sigma:=\max(0,2s(p)-(\frac{1}{2}-\delta)r)$. Then
\begin{equation}\label{eq:subcritical}
a(x,D):\HT^{s+m+\sigma,p}_{FIO}(\Rn)\to\Hps
\end{equation}
for $-r/2<s+s(p)<r$. If $a\in \Hrinf S^{m}_{1,\delta}$, then \eqref{eq:subcritical} also holds for $s+s(p)=r$. If $a\in \Hrinf (\Rn)$ and $p\in(1,\infty)$, then \eqref{eq:subcritical} holds for all $-r/2\leq s+s(p)\leq r$, with $m=\delta=0$.
\end{corollary}
\begin{proof}
As in Lemma \ref{lem:smoothing}, we write $a=a^{\sharp}_{1/2}+a^{\flat}_{1/2}$ with $a^{\sharp}_{1/2}\in S^{m}_{1,1/2}$ and $a^{\flat}_{1/2}\in C^{r}_{*}S^{m-(1/2-\delta)r}_{1,1/2}$. Then Lemma \ref{lem:smoothpseudo} yields
\[
a^{\sharp}_{1/2}(x,D):\HT^{s+m+\sigma,p}_{FIO}(\Rn)\subseteq\HT^{s+m,p}_{FIO}(\Rn)\to\Hps
\]
for all $s\in\R$. On the other hand, Proposition \ref{prop:pseudoloss} shows that
\begin{equation}\label{eq:aflatmap}
a^{\flat}_{1/2}(x,D):\HT^{s+m+\sigma,p}_{FIO}(\Rn)\subseteq\HT^{s+2s(p)+m-(1/2-\delta)r,p}_{FIO}(\Rn)\to\Hps
\end{equation}
for $-r/2<s+s(p)<r$. This proves the first statement. For the second and third statement one only has to extend \eqref{eq:aflatmap} to the cases where $s+s(p)=r$ and $s+s(p)=-r/2$, respectively. But this follows from Proposition \ref{prop:pseudoloss} as well, noting that if $a\in \Hrinf S^{m}_{1,\delta}$ then $a^{\flat}_{1/2}\in \Hrinf S^{m-(1/2-\delta)r}_{1,1/2}$, by Lemma \ref{lem:smoothing}.
\end{proof}

\begin{remark}\label{rem:Sobolevrange}
For $a\in C^{r}_{*}S^{m}_{1,\delta}$ with $\delta<1/2$, one can obtain results for $s$ in a slightly larger range than that in Corollary \ref{cor:subcritical}. To this end one chooses $\delta'\in(\delta,1/2)$ and applies the symbol smoothing procedure to write $a=a^{\sharp}_{\delta'}+a^{\flat}_{\delta'}$ with $a^{\sharp}_{\delta'}\in S^{m}_{1,\delta'}$ and $a^{\flat}_{\delta'}\in C^{r}_{*}S^{m-(\delta'-\delta)r}_{1,\delta'}$. Then Lemma \ref{lem:smoothpseudo} and Proposition \ref{prop:pseudoloss} combine to yield
\begin{equation}\label{eq:Sobolevrange}
a(x,D):\HT^{s+m+\sigma',p}_{FIO}(\Rn)\to\Hps
\end{equation}
for $-(1-\delta')r<s+s(p)<r$ and $\sigma':=\max(0,2s(p)-(\delta'-\delta)r)$. Note that $\sigma'\geq\sigma$, so that \eqref{eq:subcritical} is stronger than \eqref{eq:Sobolevrange}. However, for $p$ sufficiently close to $2$ one has $\sigma=\sigma'=0$, and then \eqref{eq:Sobolevrange} is as strong as \eqref{eq:subcritical} and holds for a bigger range of $s$.
\end{remark}

\begin{remark}\label{rem:othercoef}
In Definition \ref{def:rough} one can consider symbols $a\in \HT^{r,q}S^{m}_{1,\delta}$ with coefficients in Sobolev spaces $\HT^{r,q}(\Rn)$ for $q<\infty$, by replacing $C^{r}_{*}(\Rn)$ in \eqref{eq:symbol2} by $\HT^{r,q}(\Rn)$. The Sobolev embedding $\HT^{r,q}(\Rn)\subseteq C^{r-n/q}_{*}(\Rn)$ allows one to directly derive mapping properties for pseudodifferential operators with $\HT^{r,q}S^{m}_{1,\delta}$ symbols from the results in this article. However, slightly better results can be obtained in an alternative manner. Mapping properties for pseudodifferential operators with $\HT^{r,q}S^{m}_{1,\delta}$ symbols on $\HT^{s,p}(\Rn)$ were obtained in~\cite{Marschall88}. As in Proposition \ref{prop:pseudoloss}, one can combine such results with the Sobolev embeddings for $\Hps$ in \eqref{eq:sobolev}. Moreover, the symbol decomposition in Lemma \ref{lem:smoothing} also applies to $\HT^{r,q}S^{m}_{1,\delta}$ symbols (see \cite[Proposition 1.8.2]{Taylor00}). Just as in Corollary \ref{cor:subcritical}, one can then derive mapping properties of operators with $\HT^{r,q}S^{m}_{1,\delta}$ symbols on $\Hps$. This applies in particular to multiplication operators associated with functions $g\in \HT^{r,q}(\Rn)$.
\end{remark}

\section{The critical case}\label{sec:critical}

In Corollary \ref{cor:subcritical} we proved boundedness of pseudodifferential operators with $C^{r}_{*}S^{0}_{1,\delta}$ symbols on $\Hps$ for any $r>0$ and $\delta<1/2$, at least for $p$ sufficiently close to $2$ and $s$ sufficiently close to $-s(p)$. However, for $\delta=1/2$, Corollary \ref{cor:subcritical} only yields boundedness for $p=2$. Lemma \ref{lem:smoothpseudo} shows that a pseudodifferential operator with an $S^{0}_{1,1/2}$ symbol is bounded on $\Hps$ for all $p\in[1,\infty]$ and $s\in\R$, but that leaves open the question whether operators with $C^{r}_{*}S^{0}_{1,1/2}$ symbols, for reasonably small $r$, are bounded on $\Hps$ for $p\neq 2$. In this section we show that this is indeed the case.

The main result of this section is as follows. It shows in particular that pseudodifferential operators with $C^{r}_{*}S^{0}_{1,1/2}$ symbols are bounded on $\Hps$ for all $1<p<\infty$ and $-r/2+2s(p)<s+s(p)<r$, as long as $r>n-1$. Theorem \ref{thm:introduction} is a special case of this.

\begin{theorem}\label{thm:critical}
Let $r>0$, $m\in\R$, $p\in(1,\infty)$ and $a\in C^{r}_{*}S^{m}_{1,1/2}$. For $\veps>0$, set
\[
\tau:=\begin{cases}
0&\text{if }r>n-1,\\
\veps&\text{if }r=n-1,\\
2s(p)\big(1-\frac{r}{n-1}\big)&\text{if }r<n-1,
\end{cases}
\]
and
\[
\gamma:=\begin{cases}
\frac{1}{2}+\frac{2s(p)}{r}&\text{if }r\geq n-1,\\
\frac{1}{2}+\frac{2s(p)}{n-1}&\text{if }r<n-1.
\end{cases}
\]
Then 
\begin{equation}\label{eq:critical}
a(x,D):\HT^{s+m+\tau,p}_{FIO}(\Rn)\to \Hps
\end{equation}
for $-(1-\gamma)r<s+s(p)<r$.
If $a\in\Hrinf S^{m}_{1,1/2}$, then \eqref{eq:critical} also holds for $s+s(p)=r$. If $a=b^{\flat}_{1/2}$ for some $b\in \Hrinf (\Rn)$, then \eqref{eq:critical} holds for all $-(1-\gamma)r\leq s+s(p)\leq r$, with $m=-r/2$.
\end{theorem}
\begin{proof}
By replacing $a(x,D)$ by $a(x,D)\lb D\rb^{-m}$, we may assume that $m=0$. For the final statement this requires replacing $a(x,D)$ by $b^{\flat}_{1/2}(x,D)\lb D\rb^{r/2}$.

\subsubsection{Symbol smoothing}

We first use the symbol smoothing procedure and the results from the previous section to remove some of the frequencies of $a$. More precisely, we claim that it suffices to prove the following statement. If $a\in C^{r}_{*}S^{0}_{1,1/2}$ has the additional property that, for some $c>0$ and all $\eta\in\Rn$, one has\begin{equation}\label{eq:supportcondition}
\supp(\F a(\cdot,\eta))\subseteq\{\xi\in\Rn\mid c|\eta|^{1/2}\leq |\xi|\leq  \tfrac{1}{16}(1+|\eta|)^{\gamma}\},
\end{equation}
then \eqref{eq:critical} holds for all $s\in\R$.

To prove this claim, let a general $a\in C^{r}_{*}S^{0}_{1,1/2}$ be given and note that
\[
\tfrac{1}{2}\leq \gamma\leq \tfrac{1}{2}+\tfrac{2s(p)}{n-1} <\tfrac{1}{2}+\tfrac{n-1}{2(n-1)}=1
\]
and
\[
\tau\geq 2s(p)-(\gamma-\tfrac{1}{2})r.
\]
We apply the symbol smoothing procedure in Lemma \ref{lem:smoothing} twice, once to $a$ and once to $a^{\flat}_{1/2}$, to write
\begin{equation}\label{eq:symboldecomp}
a=a^{\sharp}_{1/2}+a^{\flat}_{1/2}=a^{\sharp}_{1/2}+(a^{\flat}_{1/2})^{\sharp}_{\gamma}+(a^{\flat}_{1/2})^{\flat}_{\gamma}.
\end{equation}
Then $a^{\sharp}_{1/2}\in S^{0}_{1,1/2}$ and $(a^{\flat}_{1/2})^{\flat}_{\gamma}\in C^{r}_{*}S^{-(\gamma-1/2)r}_{1,\gamma}$, by Lemma \ref{lem:smoothing}. Hence 
\begin{equation}\label{eq:asharpbound}
a^{\sharp}_{1/2}(x,D):\HT^{s+\tau,p}_{FIO}(\Rn)\subseteq\Hps\to\Hps
\end{equation}
for all $s\in\R$, by Lemma \ref{lem:smoothpseudo}, and
\begin{equation}\label{eq:doubleflat}
(a^{\flat}_{1/2})^{\flat}_{\gamma}(x,D):\HT^{s+\tau,p}_{FIO}(\Rn)\subseteq\HT^{s+2s(p)-(\gamma-1/2)r,p}_{FIO}(\Rn)\to\Hps
\end{equation}
for $-(1-\gamma)r<s+s(p)<r$, by Proposition \ref{prop:pseudoloss}. Moreover, if $a\in\Hrinf S^{0}_{1,1/2}$ then $(a^{\flat}_{1/2})^{\flat}_{\gamma}\in \Hrinf S^{-(\gamma-1/2)r}_{1,\gamma}$, again by Lemma \ref{lem:smoothing}, and then Proposition \ref{prop:pseudoloss} shows that \eqref{eq:doubleflat} also holds for $s+s(p)=r$. Finally, if $a(x,D)=b_{1/2}^{\flat}(x,D)\lb D\rb^{r/2}$ for some $b\in \Hrinf (\Rn)$, then Remark \ref{rem:doubleflat} shows that \eqref{eq:doubleflat} also holds for $s+s(p)=-(1-\gamma)r$.

It thus follows from \eqref{eq:symboldecomp}, \eqref{eq:asharpbound} and \eqref{eq:doubleflat} that it suffices to prove \eqref{eq:critical} for all $s\in\R$ with $a(x,D)$ replaced by $(a^{\flat}_{1/2})^{\sharp}_{\gamma}(x,D)$. It is straightforward to check that, for $\ph$ supported sufficiently close to zero, $(a^{\flat}_{1/2})^{\sharp}_{\gamma}$ has the property in \eqref{eq:supportcondition} for some $c>0$ and all $\eta\in\Rn$:
\[
\supp\big(\F \big((a^{\flat}_{1/2})^{\sharp}_{\gamma}(\cdot,\eta)\big)\big)\subseteq\{\xi\in\Rn\mid c|\eta|^{1/2}\leq |\xi|\leq  \tfrac{1}{16}(1+|\eta|)^{\gamma}\}.
\]
Moreover, $a^{\flat}_{1/2}\in C^{r}_{*}S^{0}_{1,1/2}$ and $(a^{\flat}_{1/2})^{\flat}_{\gamma}\in C^{r}_{*}S^{-(\gamma-1/2)r}_{1,\gamma}\subseteq C^{r}_{*}S^{0}_{1,1/2}$, so $(a^{\flat}_{1/2})^{\sharp}_{\gamma}\in C^{r}_{*}S^{0}_{1,1/2}$ as well, by \eqref{eq:symboldecomp}. This proves the claim, and the remainder of the proof is dedicated to showing \eqref{eq:critical} for all $s\in\R$ and $a\in C^{r}_{*}S^{0}_{1,1/2}$ satisfying \eqref{eq:supportcondition}.

\subsubsection{A symbol decomposition} 

Next, we will use a symbol decomposition as in~\cite[Chapter 2]{Coifman-Meyer78} (see also~\cite{Bourdaud82,Marschall88,Taylor00}). Let $(\psi_{k})_{k=0}^{\infty}$ be the Littlewood-Paley decomposition from \eqref{eq:LitPaldecomp}, and suppose that $\psi_{1}(\eta)=0$ for $|\eta|\notin [(1+\veps)/2,2-\veps]$ for some $\veps>0$. Let $(\wt{\psi}_{k})_{k=0}^{\infty}\subseteq C^{\infty}_{c}(\Rn)$ be such that $\wt{\psi}_{k}\equiv 1$ on $\supp(\psi_{k})$ for each $k\geq0$, $\wt{\psi}_{0}(\eta)=0$ for $|\eta|>2$, and such that for all $k\geq1$ and $\eta\in\Rn$ one has $\wt{\psi}_{1}(\eta)=0$ if $|\eta|\notin [1/2,2]$, and $\wt{\psi}_{k}(\eta)=\wt{\psi}_{1}(2^{-k+1}\eta)$. By decomposing the function 
\[
\zeta\mapsto a(x,2^{k+1}\pi\zeta)\psi_{k}(2^{k+1}\pi\zeta)
\]
on $[-1/2,1/2]^{n}$ into its Fourier modes, we can write
\[
a(x,\eta)\psi_{k}(\eta)=\sum_{\beta\in\Z^{n}}c_{k,\beta}(x)e^{i\beta2^{-k}\eta}\wt{\psi}_{k}(\eta)
\]
for all $k\geq0$ and $x,\eta\in\Rn$, where
\begin{equation}\label{eq:defck}
c_{k,\beta}(x)=\int_{[-1/2,1/2]^{n}}e^{-2\pi i \beta\zeta}a(x,2^{k+1}\pi\zeta)\psi_{k}(2^{k+1}\pi\zeta)\ud\zeta.
\end{equation}
Then 
\begin{equation}\label{eq:Fourdecomp}
a(x,\eta)=\sum_{k=0}^{\infty}a(x,\eta)\psi_{k}(\eta)=\sum_{\beta\in\Z^{n}}\sum_{k=0}^{\infty}c_{k,\beta}(x)e^{i\beta2^{-k}\eta}\wt{\psi}_{k}(\eta).
\end{equation}
Here the $c_{k,\beta}$ decay rapidly as $|\beta|\to\infty$, as can be seen by integrating by parts in \eqref{eq:defck} and using that $a\in C^{r}_{*}S^{0}_{1,1/2}$ and that $\psi_{k}(\eta)=\psi_{1}(2^{-k+1}\eta)=0$ if $|\eta|\notin[2^{k-2},2^{k}]$ and $k\geq1$. Moreover, for each $k\geq0 $ one has $\F c_{k,\beta}(\xi)=0$ if $|\xi|<c2^{(k-2)/2}$ or $|\xi|>2^{k\gamma-3}\geq \tfrac{1}{16}(1+2^{k})^{\gamma}$, as follows from assumption \eqref{eq:supportcondition} and the support condition on $\psi_{k}$. 

By inspecting the properties of each of the series in \eqref{eq:Fourdecomp} for fixed $\beta\in\Z_{+}^{n}$, using in particular the rapid decay of the $c_{k,\beta}$ and that $a\in C^{r}_{*}S^{0}_{1,1/2}$, it follows that we may suppose in the remainder that $a$ has the form
\begin{equation}\label{eq:anew}
a(x,\eta)=\sum_{k=0}^{\infty}a_{k}(x)\chi_{k}(\eta)\quad(x,\eta\in \Rn),
\end{equation}
where $(a_{k})_{k=0}^{\infty}\subseteq C^{r}_{*}(\Rn)$ satisfies
\begin{equation}\label{eq:abounds}
\sup_{k\geq0}\|a_{k}\|_{L^{\infty}(\Rn)}+2^{-kr/2}\|a_{k}\|_{C^{r}_{*}(\Rn)}<\infty
\end{equation}
and, for all $k\geq0$ and $\eta\in\Rn$,
\begin{equation}\label{eq:supportak}
\supp(\F a_{k})\subseteq\{\xi\in\Rn\mid c2^{(k-2)/2}\leq |\xi|\leq  2^{k\gamma-3}\}.
\end{equation}
The latter assumption implies that in fact $(a_{k})_{k=0}^{\infty}\subseteq C^{\infty}(\Rn)$. Moreover, $(\chi_{k})_{k=0}^{\infty}\subseteq C^{\infty}_{c}(\Rn)$ is such that $\chi_{0}(\eta)=0$ if $|\eta|>1$, $\chi_{1}(\eta)=0$ if $|\eta|\notin [1/2,2]$, and $\chi_{k}(\eta)=\chi_{1}(2^{-k+1}\eta)$ for all $k>1$ and $\eta\in\Rn$. 

\subsubsection{The low-frequency components} 

After this preliminary work, we now get to the heart of the proof. Let $f\in\Sw(\Rn)$. For $a$ given by \eqref{eq:anew}, by \eqref{eq:HpFIOnorm} we need to show that
\begin{equation}\label{eq:newertoprove}
\|q(D)a(x,D)f\|_{L^{p}}+\Big(\int_{S^{n-1}}\|\ph_{\w}(D)a(x,D)f\|_{\HT^{s,p}}^{p}\ud\w\Big)^{1/p}\lesssim \|f\|_{\HT^{s+\tau,p}_{FIO}}
\end{equation}
for an implicit constant independent of $f$. Throughout, for $k\geq0$ write $f_{k}:=\chi_{k}(D)f$.

We first deal with the low-frequency component $q(D)a(x,D)f$. Here one can directly combine the smoothing property of $q(D)$ with Lemma \ref{lem:pseudoLp} and the Sobolev embeddings in \eqref{eq:sobolev}, but instead we make an observation which will also be useful later on. For each $k\geq1$ one has
\begin{equation}\label{eq:vanish}
\F(a_{k}f_{k})(\xi)=0\quad (|\xi|\notin [2^{k-3},2^{k+1}]),
\end{equation} 
as follows from the support properties of $\F(a_{k})$ and $\F(f_{k})=\chi_{k}\F(f)$. Indeed,
\begin{equation}\label{eq:conv}
\supp(\F(a_{k}f_{k}))=\supp(\F(a_{k})\ast \F(f_{k}))\subseteq\supp(\F(a_{k}))+\supp(\F(f_{k})),
\end{equation}
and for $\zeta\in \supp(\F(a_{k}))$ and $\eta\in\supp(\F(f_{k}))$ one has $|\zeta|\leq 2^{k\gamma-3}\leq 2^{k-3}$, $|\eta|\in[2^{k-2},2^{k}]$ and 
\begin{equation}\label{eq:sumbound}
2^{k-3}=2^{k-2}-2^{k-3}\leq |\zeta+\eta|\leq 2^{k}+2^{k-3}\leq 2^{k+1}.
\end{equation}
Since $q(\xi)=0$ for $|\xi|\geq2$, it follows from \eqref{eq:vanish} that
\[
q(D)a(x,D)f=\sum_{k=0}^{3}q(D)(a_{k}f_{k}).
\]
Now the fact that $q\in C^{\infty}_{c}(\Rn)$, the assumption that $(a_{k})_{k=0}^{\infty}\subseteq C^{\infty}(\Rn)$ and $(\chi_{k})_{k=0}^{\infty}\subseteq C^{\infty}_{c}(\Rn)$, and the Sobolev embeddings from \eqref{eq:sobolev} yield
\begin{align*}
\|q(D)a(x,D)f\|_{L^{p}}&\leq\sum_{k=0}^{3} \|q(D)(a_{k}f_{k})\|_{L^{p}}\lesssim \sum_{k=0}^{3} \|a_{k}f_{k}\|_{\HT^{s+\tau-s(p),p}}\\
&\lesssim \sum_{k=0}^{3} \|\chi_{k}(D)f\|_{\HT^{s+\tau-s(p),p}}\lesssim \|f\|_{\HT^{s+\tau-s(p),p}}\lesssim \|f\|_{\HT^{s+\tau,p}_{FIO}}.
\end{align*}
This proves half of \eqref{eq:newertoprove}.

For the other half of \eqref{eq:newertoprove} first note that, since $q\in C^{\infty}_{c}(\Rn)$, there exists an $N\geq0$ such that $a(x,D)q(D)f=\sum_{k=0}^{N}a_{k}\chi_{k}(D)q(D)f$. By considering functions with a single nonzero coordinate, it follows in particular from Lemma \ref{lem:multiplier} that $\ph_{\w}(D):\HT^{s+\frac{n-1}{4},p}(\Rn)\to\HT^{s,p}(\Rn)$ uniformly in $\w\in S^{n-1}$. Hence we obtain in the same manner as before that
\begin{align*}
&\|\ph_{\w}(D)a(x,D)q(D)f\|_{\HT^{s,p}}\leq\sum_{k=0}^{N} \|\ph_{w}(D)(a_{k}q(D)\chi_{k}(D)f)\|_{\HT^{s,p}}\\
&\lesssim \sum_{k=0}^{N} \|a_{k}q(D)\chi_{k}(D)f\|_{\HT^{s+(n-1)/4,p}}\lesssim \sum_{k=0}^{N} \|q(D)\chi_{k}(D)f\|_{\HT^{s+(n-1)/4,p}}\\
&\lesssim \sum_{k=0}^{N} \|f\|_{\HT^{s+\tau-s(p),p}}\lesssim \|f\|_{\HT^{s+\tau-s(p),p}}\lesssim \|f\|_{\HT^{s+\tau,p}_{FIO}},
\end{align*}
for implicit constants independent of $\w$ and $f$. Hence
\[
\Big(\int_{S^{n-1}}\|\ph_{\w}(D)a(x,D)q(D)f\|_{\HT^{s,p}}^{p}\ud\w\Big)^{1/p}\lesssim \|f\|_{\HT^{s+\tau,p}_{FIO}},
\]
and for \eqref{eq:newertoprove} it remains to show that 
\[
\Big(\int_{S^{n-1}}\|\ph_{\w}(D)a(x,D)(1-q)(D)f\|_{\HT^{s,p}}^{p}\ud\w\Big)^{1/p}\lesssim \|f\|_{\HT^{s+\tau,p}_{FIO}}.
\]
Moreover, since $q(D):\HT^{s+\tau,p}_{FIO}(\Rn)\to\HT^{s+\tau,p}_{FIO}(\Rn)$ is bounded, as follows e.g.~from \eqref{eq:sobolev} or Lemma \ref{lem:smoothpseudo}, it suffices to show that
\[
\Big(\int_{S^{n-1}}\|\ph_{\w}(D)a(x,D)(1-q)(D)f\|_{\HT^{s,p}}^{p}\ud\w\Big)^{1/p}\lesssim \|(1-q)(D)f\|_{\HT^{s+\tau,p}_{FIO}}.
\]
This is what we will focus on in the remainder of the proof. However, for simplicity of notation and because the only property that we will use of $(1-q)(D)f$ is that $\chi_{0}(D)(1-q)(D)f=0$, we will continue working with $f$ and prove more generally that
\begin{equation}\label{eq:newesttoprove}
\Big(\int_{S^{n-1}}\|\ph_{\w}(D)a(x,D)f\|_{\HT^{s,p}}^{p}\ud\w\Big)^{1/p}\lesssim \|f\|_{\HT^{s+\tau,p}_{FIO}}
\end{equation}
under the additional assumption that $\chi_{0}(D)f=0$.

\subsubsection{The high-frequency component} 

In the remainder of the proof we will show \eqref{eq:newesttoprove} under the additional assumption that $f_{0}=\chi_{0}(D)f=0$. To this end, write $a_{k,j}:=\psi_{j}(D)a_{k}$ for $k,j\geq0$, and $a_{k,j}:=0$ for $j<0$. Note that, by \eqref{eq:supportak}, there exists an $M\geq0$ such that
\[
a(x,D)f=\sum_{k=1}^{\infty}\sum_{j=-M+k/2}^{k\gamma-2}a_{k,j}f_{k}.
\]
For convenience, we assume that $M\geq3$. Let $\w\in S^{n-1}$. Then Lemma \ref{lem:LitPal} and \eqref{eq:vanish} yield
\begin{equation}\label{eq:LitPaluse}
\begin{aligned}
&\|\ph_{\w}(D)a(x,D)f\|_{\HT^{s,p}}\\
&=\Big\|\sum_{k=1}^{\infty}\lb D\rb^{-\frac{n-1}{4}}\ph_{\w}(D)\Big(\sum_{j=-M+k/2}^{k\gamma-2}a_{k,j}f_{k}\Big)\Big\|_{\HT^{s+(n-1)/4,p}}\\
&\lesssim \Big\|\Big(\sum_{k=1}^{\infty}4^{k(s+\frac{n-1}{4})}\Big|\lb D\rb^{-\frac{n-1}{4}}\ph_{\w}(D)\Big(\sum_{j=-M+k/2}^{k\gamma-2}a_{k,j}f_{k}\Big)\Big|^{2}\Big)^{1/2}\Big\|_{L^{p}}
\end{aligned}
\end{equation}
for an implicit constant independent of $\w$. Next, as in \eqref{eq:phiwintegral}, let $m\in S^{(n-1)/4}(\Rn)$ be such that $f_{k}=\int_{S^{n-1}}m(D)\ph_{\nu}(D)f_{k}\,\ud\nu$ for all $k\geq1$, which is possible since $\wh{f}_{k}(\zeta)=0$ for $|\zeta|<1/2$. Write $f_{k,\nu}:=m(D)\ph_{\nu}(D)f_{k}$ for $\nu\in S^{n-1}$. We claim that 
\begin{equation}\label{eq:restrict}
\ph_{\w}(D)\Big(\sum_{j=-M+k/2}^{k\gamma-2}a_{k,j}f_{k}\Big)=\ph_{\w}(D)\Big(\sum_{j=-M+k/2}^{k\gamma-2}a_{k,j}\int_{F_{\w,k,j}}f_{k,\nu}\,\ud\nu\Big),
\end{equation}
where 
\begin{equation}\label{eq:defF}
F_{\w,k,j}:=\{\nu \in S^{n-1}\mid |\nu-\w|\leq 2^{3+M+j-k}\}
\end{equation}
for $-M+k/2\leq j\leq k\gamma-2$. Note that $|F_{\w,k,j}|\eqsim 2^{(n-1)(j-k)}$.

To prove \eqref{eq:restrict}, let $k\geq1$, $-M+k/2\leq j\leq k\gamma-2$, $\nu\in S^{n-1}$, $\zeta\in \supp(\F(a_{k,j}))$ and $\eta\in\supp(\F(f_{k,\nu}))$ be given. Then $|\hat{\eta}-\nu|\leq 2^{1-k/2}$ and, using \eqref{eq:sumbound},
\begin{equation}\label{eq:fracbound}
\frac{\big||\zeta+\eta|-|\eta|\big|}{|\zeta+\eta|}\leq \frac{|\zeta|}{|\zeta+\eta|}\leq \frac{2^{j}}{2^{k-3}}.
\end{equation}
Hence
\[
\Big|\frac{\eta}{|\zeta+\eta|}-\nu\Big|\leq |\hat{\eta}-\nu|+\Big|\frac{\eta}{|\zeta+\eta|}-\frac{\eta}{|\eta|}\Big|\leq 2^{1-k/2}+\frac{\big||\zeta+\eta|-|\eta|\big|}{|\zeta+\eta|}\leq 2^{1-k/2}+2^{3+j-k}
\]
and, using \eqref{eq:fracbound} again,
\begin{align*}
\Big|\frac{\zeta+\eta}{|\zeta+\eta|}-\nu\Big|&\leq \Big|\frac{\eta}{|\zeta+\eta|}-\nu\Big|+\frac{|\zeta|}{|\zeta+\eta|}\leq 2^{1-k/2}+2^{4+j-k}=2^{1-k+k/2}+2^{4+j-k}\\
&\leq 2^{1+M+j-k}+2^{4+j-k}\leq 2^{2+M+j-k},
\end{align*}
where we also used that $M\geq3$. If $\nu\notin F_{\w,k,j}$ then this in turn implies that 
\begin{align*}
\Big|\frac{\zeta+\eta}{|\zeta+\eta|}-\w\Big|&\geq |\nu-\w|-\Big|\frac{\zeta+\eta}{|\zeta+\eta|}-\nu\Big|>2^{3+M+j-k}-2^{2+M+j-k}\\
&=2^{2+M+j-k}\geq 2^{2-k+k/2}\geq 2^{1-k/2}\geq |\zeta+\eta|^{-1/2},
\end{align*}
so that $\zeta+\eta\notin \supp(\ph_{\w})$, by \eqref{eq:phiwsupport}. By \eqref{eq:conv} with $a_{k}$ replaced by $a_{k,j}$ and $f_{k}$ replaced by $f_{k,\nu}$, this proves \eqref{eq:restrict}.

Now we can combine \eqref{eq:LitPaluse}, \eqref{eq:restrict} and Lemma \ref{lem:multiplier} to write
\begin{align*}
&\|\ph_{\w}(D)a(x,D)f\|_{\HT^{s,p}}\\
&\lesssim \Big\|\Big(\sum_{k=1}^{\infty}4^{k(s+\frac{n-1}{4})}\Big|\lb D\rb^{-\frac{n-1}{4}}\ph_{\w}(D)\Big(\sum_{j=-M+k/2}^{k\gamma-2}a_{k,j}f_{k}\Big)\Big|^{2}\Big)^{1/2}\Big\|_{L^{p}}\\
&=\Big\|\Big(\sum_{k=1}^{\infty}4^{k(s+\frac{n-1}{4})}\Big|\lb D\rb^{-\frac{n-1}{4}}\ph_{\w}(D)\Big(\sum_{j=-M+k/2}^{k\gamma-2}a_{k,j}\int_{F_{\w,k,j}}f_{k,\nu}\,\ud\nu\Big)\Big|^{2}\Big)^{1/2}\Big\|_{L^{p}}\\
&\lesssim \Big\|\Big(\sum_{k=1}^{\infty}4^{k(s+\frac{n-1}{4})}\Big|\sum_{j=-M+k/2}^{k\gamma-2}a_{k,j}\int_{F_{\w,k,j}}f_{k,\nu}\,\ud\nu\Big|^{2}\Big)^{1/2}\Big\|_{L^{p}}\\
&\lesssim \Big\|\Big(\sum_{k=1}^{\infty}4^{k(s+\frac{n-1}{4})}\Big(\sum_{j=-M+k/2}^{k\gamma-2}|a_{k,j}|2^{(n-1)(j-k)}\fint_{F_{\w,k,j}}|f_{k,\nu}|\ud\nu\Big)^{2}\Big)^{1/2}\Big\|_{L^{p}}.
\end{align*}
Next, by \eqref{eq:abounds}, one has
\[
\|a_{k,j}\|_{L^{\infty}}=\|\psi_{j}(D)a_{k}\|_{L^{\infty}}\leq 2^{-jr}\|a_{k}\|_{C^{r}_{*}}\lesssim 2^{(\frac{k}{2}-j)r}
\]
for all $k,j\geq0$. Hence 
\begin{align*}
&\|\ph_{\w}(D)a(x,D)f\|_{\HT^{s,p}}\\
&\lesssim \Big\|\Big(\sum_{k=1}^{\infty}4^{k(s+\frac{n-1}{4})}\Big(\sum_{j=-M+k/2}^{k\gamma-2}2^{k(\frac{r}{2}+1-n)+j(n-1-r)}\fint_{F_{\w,k,j}}|f_{k,\nu}|\ud\nu\Big)^{2}\Big)^{1/2}\Big\|_{L^{p}}\\
&=\Big\|\Big(\sum_{k=1}^{\infty}\Big(\sum_{j=-M+k/2}^{k\gamma-2}2^{k(\frac{r}{2}-\tau-\frac{n-1}{2})+j(n-1-r)}\fint_{F_{\w,k,j}}2^{k(s+\tau-\frac{n-1}{4})}|f_{k,\nu}|\ud\nu\Big)^{2}\Big)^{1/2}\Big\|_{L^{p}}.
\end{align*}
Let $M$ be the centered Hardy--Littlewood maximal function on $S^{n-1}$, given for $g\in L^{1}(S^{n-1})$ by
\[
Mg(\w):=\sup_{B}\fint_{B}|g(\nu)|\ud \nu\quad(\w\in S^{n-1}),
\]
where the supremum is taken over all balls $B\subseteq S^{n-1}$ with center $\w$. Also, for $x\in\Rn$ and $k\geq1$ set 
\[
g_{k,x}(\nu):=2^{k(s+\tau-\frac{n-1}{4})}f_{k,\nu}(x)\quad(\nu\in S^{n-1}).
\]
Now, $\tau$ and $\gamma$ are chosen so that 
\[
\sum_{j=-M+k/2}^{k\gamma-2}2^{k(\frac{r}{2}-\tau-\frac{n-1}{2})+j(n-1-r)}\lesssim 1,
\]
as is straightforward to check. We can combine all this with the boundedness of $M$ on $L^{p}(S^{n-1};\ell^{2})$ (see~\cite[Section 6.6]{GeGoKoKr98}) to obtain
\begin{align*}
&\int_{S^{n-1}}\|\ph_{\w}(D)a(x,D)f\|_{\HT^{s,p}}^{p}\ud \w\\
&\lesssim \int_{S^{n-1}}\int_{\Rn}\!\Big(\sum_{k=1}^{\infty}\Big(\sum_{j=-M+k/2}^{k\gamma-2}2^{k(\frac{r}{2}-\tau-\frac{n-1}{2})+j(n-1-r)}\fint_{F_{\w,k,j}}|g_{k,x}(\nu)|\ud\nu\Big)^{2}\Big)^{p/2}\ud x\ud \w\\
&\leq \int_{S^{n-1}}\int_{\Rn}\!\Big(\sum_{k=1}^{\infty}\Big(\sum_{j=-M+k/2}^{k\gamma-2}2^{k(\frac{r}{2}-\tau-\frac{n-1}{2})+j(n-1-r)}Mg_{k,x}(\w)\Big)^{2}\Big)^{p/2}\ud x\ud \w\\
&\lesssim \int_{\Rn}\int_{S^{n-1}}\!\Big(\sum_{k=1}^{\infty}|Mg_{k,x}(\w)|^{2}\Big)^{p/2}\ud \w\ud x\lesssim \int_{\Rn}\int_{S^{n-1}}\!\Big(\sum_{k=1}^{\infty}|g_{k,x}(\nu)|^{2}\Big)^{p/2}\ud\nu \ud x\\
&=\int_{S^{n-1}}\Big\|\Big(\sum_{k=1}^{\infty}4^{k(s+\tau-\frac{n-1}{4})}|\chi_{k}(D)m(D)\ph_{\nu}(D)f|^{2}\Big)^{1/2}\Big\|_{L^{p}}^{p}\ud\nu.
\end{align*}
Now we apply Lemma \ref{lem:LitPal} to this expression, and use that $m\in S^{(n-1)/4}(\Rn)$ and that $\sum_{k=0}^{\infty}\chi_{k}\in S^{0}(\Rn)$, as is straightforward to check, to obtain 
\begin{align*}
&\int_{S^{n-1}}\|\ph_{\w}(D)a(x,D)f\|_{\HT^{s,p}}^{p}\ud \w\\
&\lesssim \int_{S^{n-1}}\Big\|\sum_{k=1}^{\infty}m(D)\chi_{k}(D)\ph_{\nu}(D)f\Big\|_{\HT^{s+\tau-(n-1)/4,p}}^{p}\ud\nu\\
&\lesssim \int_{S^{n-1}}\|\ph_{\nu}(D)f\|_{\HT^{s+\tau,p}}^{p}\ud\nu\lesssim \|f\|_{\HT^{s+\tau,p}_{FIO}}^{p}.
\end{align*}
This proves the second half of \eqref{eq:newertoprove} and concludes the proof.
\end{proof}

\begin{remark}\label{rem:dependence}
It follows from the proof of Theorem \ref{thm:critical} that the operator norm of $a(x,D)$ depends on $a$ through a finite number of the $C^{r}_{*}S^{m}_{1,\delta}$ seminorms of $a$. More precisely, for all $p\in(1,\infty)$ and $-(1-\gamma)r<s+s(p)<r$ one has 
\[
\|a(x,D)\|_{\La(\HT^{s+m+\tau,p}_{FIO}(\Rn),\Hps)}\lesssim \max_{|\alpha|\leq N}C_{\alpha},
\]
with an implicit constant independent of $a\in C^{r}_{*}S^{m}_{1,\delta}$. Here $C_{\alpha}$, for $\alpha\in\Z_{+}^{n}$, is such that \eqref{eq:symbol1} and \eqref{eq:symbol2} hold. Similar statements apply for the extremal values $s+s(p)=-(1-\gamma)r$ and $s+s(p)=r$ in Theorem \ref{thm:critical}, as well as to other results in this article such as Proposition \ref{prop:pseudoloss} and Corollary \ref{cor:subcritical}. 
\end{remark}

\begin{remark}\label{rem:p1}
For $p=1$ one can still rely on Littlewood-Paley theory as in Lemma \ref{lem:LitPal}. However, there are two obstacles in extending our proof of Theorem \ref{thm:critical} to the case where $p=1$. Firstly, Lemma \ref{lem:multiplier} played a crucial role in the proof, and it is unclear to the author whether that lemma extends (in some form) to the case where $p=1$. Secondly, we used the assumption that $1<p<\infty$ when relying on the boundedness of the Hardy--Littlewood maximal function on $L^{p}(S^{n-1};\ell^{2})$. We choose to not pursue this matter any further in the present article.
\end{remark}

\begin{remark}\label{rem:smooth}
Theorem \ref{thm:critical} recovers the statement of Lemma \ref{lem:smoothpseudo}. One might note that we in fact used Lemma \ref{lem:smoothpseudo} in the proof of Theorem \ref{thm:critical} to deal with the low frequencies of $a$, captured by $a^{\sharp}_{1/2}$. However, it is not too difficult to adapt the proof of Theorem \ref{thm:critical} to also deal with $a^{\sharp}_{1/2}$ directly. This involves a modification of the definition of $F_{\w,k,j}$ in \eqref{eq:defF}, by setting $F_{\w,k,j}:=\{\nu\in S^{n-1}\mid |\nu-\w|\leq \kappa 2^{-k/2}\}$ for $j<-M+k/2$ and a suitable $\kappa>0$. Moreover, one has to use the supremum bounds for $(a_{k})_{k=0}^{\infty}$ in \eqref{eq:abounds}. Then similar arguments yield an alternative proof of Lemma \ref{lem:smoothpseudo}.
\end{remark}

Through an application of Lemma \ref{lem:smoothing}, one obtains the following generalization of Theorem \ref{thm:critical}. Note that \eqref{eq:combined} is an improvement of \eqref{eq:subcritical}, although it holds for a slightly smaller range of $s$. 

\begin{corollary}\label{cor:combined}
Let $r>0$, $m\in\R$, $\delta\in[0,1/2]$, $p\in(1,\infty)$ and $a\in C^{r}_{*}S^{m}_{1,\delta}$. For $\veps>0$, set 
\[
\rho:=\begin{cases}
0&\text{if }2s(p)(1-\frac{r}{n-1})\leq(\frac{1}{2}-\delta)r\text{ and }\delta\neq \frac{1}{2},\\
0&\text{if }r>n-1\text{ and }\delta=\frac{1}{2},\\
\veps&\text{if }r=n-1\text{ and }\delta=\frac{1}{2},\\
2s(p)(1-\frac{r}{n-1})-(\frac{1}{2}-\delta)r&\text{otherwise},
\end{cases}
\]
and let $\gamma$ be as in Theorem \ref{thm:critical}. Then 
\begin{equation}\label{eq:combined}
a(x,D):\HT^{s+m+\rho,p}_{FIO}(\Rn)\to \Hps
\end{equation}
for $-(1-\gamma)r<s+s(p)<r$. If $a\in\Hrinf S^{0}_{1,\delta}$, then \eqref{eq:combined} also holds for $s+s(p)=r$. If $a\in \Hrinf (\Rn)$, then \eqref{eq:combined} holds for all $-(1-\gamma)r\leq s+s(p)\leq r$, with $m=\delta=0$.
\end{corollary}
Note that $\rho=\max(0,\tau-(\frac{1}{2}-\delta)r)$, where $\tau$ is as in Theorem \ref{thm:critical}.
\begin{proof}
Use Lemma \ref{lem:smoothing} to write $a=a^{\sharp}_{1/2}+a^{\flat}_{1/2}$ with $a^{\sharp}_{1/2}\in S^{m}_{1,1/2}$ and $a^{\flat}_{1/2}\in C^{r}_{*}S^{m-(1/2-\delta)r}_{1,1/2}$. By Lemma \ref{lem:smoothpseudo}, or Theorem \ref{thm:critical}, one has
\[
a^{\sharp}_{1/2}(x,D):\HT^{s+m+\rho,p}_{FIO}(\Rn)\subseteq\HT^{s+m,p}_{FIO}(\Rn)\to\Hps
\]
for all $s\in\R$. Moreover, by Theorem \ref{thm:critical},
\[
a^{\flat}_{1/2}(x,D):\HT^{s+m+\rho,p}_{FIO}(\Rn)\subseteq\HT^{s+m-(1/2-\delta)r+\tau,p}_{FIO}(\Rn)\to\Hps
\]
for $-(1-\gamma)r<s+s(p)<r$. The final two statements follow from the final two statements in Theorem \ref{thm:critical}, applied to $a^{\flat}_{1/2}$.
\end{proof}

\section*{Acknowledgments}

The author would like to thank Andrew Hassell for many useful conversations about the article, and both Andrew Hassell and Pierre Portal for valuable advice. The author is also grateful to Dorothee Frey for a conversation regarding the use of the anisotropic multiplier theorem in Lemma \ref{lem:multiplier}, and to the anonymous referee for various helpful comments.

\bibliographystyle{plain}
\bibliography{Bibliography}

\end{document}